\documentclass{article}

\usepackage{arxiv}

\usepackage{cite}
\usepackage[utf8]{inputenc} 
\usepackage[T1]{fontenc}    
\usepackage{url}            
\usepackage{booktabs}       
\usepackage{amsfonts}       
\usepackage{nicefrac}       
\usepackage{microtype}      
\usepackage{mathtools}
\usepackage{amsmath,amssymb,amsfonts}
\usepackage{graphicx}
\usepackage{textcomp}
\usepackage{xcolor}
\usepackage{hyperref}
\usepackage{amsbsy}
\usepackage{amsthm}
\usepackage{enumitem}
\usepackage{bbm}
\usepackage{algorithmic}
\usepackage[]{algorithm}
\usepackage{subcaption}
\usepackage{soul}
\usepackage{lettrine}
\usepackage[symbol]{footmisc}
\usepackage{mathtools}
\usepackage{caption}
\newcommand{\abs}[1]{\ensuremath{\left\lvert#1\right\rvert}}
\renewcommand{\b}[1]{\ensuremath{\mathbf{#1}}} 
\renewcommand{\c}[1]{\ensuremath{\mathcal{#1}}} 
\newcommand{\ind}{1\hspace{-1.6mm}1} 
\DeclareMathOperator{\prox}{\mathbf{prox}}

\newcommand{\norm}[1]{\ensuremath{\left\|#1\right\|}} 
\newcommand{\tb}[1]{\ensuremath{\tilde{\mathbf{#1}}}} 
\newcommand{\eqtext}[1]{\ensuremath{\stackrel{\text{#1}}{=}}} 
\newcommand{\leqtext}[1]{\ensuremath{\stackrel{\text{#1}}{\leq}}} 
\newcommand{\geqtext}[1]{\ensuremath{\stackrel{\text{#1}}{\geq}}} 
\DeclareMathOperator*{\argmin}{arg\,min} 
\providecommand{\ip}[2]{\langle #1, #2 \rangle} 
\newcommand{\px}[2]{\ensuremath{\mathcal{P}_{#1}\left(#2\right)}} 

\def \sn {\sum_{t=1}\sp{N}}
\def \Reg {{\mathrm{Reg}}}

\def \x {{\boldsymbol{x}}}
\def \v {{\boldsymbol{v}}}
\def \w {{\boldsymbol{w}}}
\def \y {{\boldsymbol{y}}}
\def \u {{\boldsymbol{u}}}
\def \z {{\boldsymbol{z}}}
\def \nn {\nonumber}

\def \tx {{\tb{\boldsymbol{x}}}}

\def \Jt {{\tilde{J}}}
\def \pn {{\mathcal{P}_N}}

\def \I {{\b{I}}}

\def \T {{\mathsf{T}}}

\def \O {{\c{O}}}

\def \cA {{\c{A}}}

\def \cX {{\c{X}}}
\def \bX {{\boldsymbol{\mathcal{X}}}}

\def \cN {{\c{N}}}

\def \Rn {{\mathbb{R}}}

\renewcommand{\^}[1]{\ensuremath{\sp{(#1)}}}		

\newtheorem{assumption}{}

\theoremstyle{remark}
\newtheorem{rem}{\bf Remark}
\newtheorem{theorem}{Theorem}
\newtheorem{lemma}{Lemma}

\newtheorem{corollary}[theorem]{Corollary}

\hypersetup{
    colorlinks=true,
    linkcolor=blue,
    citecolor=blue,
    urlcolor=black
}

\begin{document}

\title{Proximal Algorithms for Smoothed Online Convex Optimization with Predictions} 

\author{
	Spandan Senapati\(^*\) \\
        Indian Institute of Technology Kanpur \\
	\And
	Ashwin Shenai\(^\dagger\) \\
        Indian Institute of Technology Kanpur \\
	\\
	\And
	Ketan Rajawat\(^\ddagger\)  \\
	Indian Institute of Technology Kanpur \\
}

\maketitle
\footnotetext{

* Currently with the Department of Computer Science, University of Southern California, Los Angeles, CA, USA (email: \texttt{ssenapat@usc.edu}).

\(\dagger\) Currently with the Department of Computer Science (D-INFK), ETH Zurich, Switzerland (email: \texttt{ashenai@ethz.ch}).

\(\ddagger\) Department of Electrical Engineering (email: \texttt{ketan@iitk.ac.in})

This work was done while the authors were at the Indian Institute of Technology Kanpur.
}

\begin{abstract}
	We consider a smoothed online convex optimization (SOCO) problem with predictions, where the learner has access to a finite lookahead window of time-varying stage costs, but suffers a {switching cost} for changing its actions at each stage. Based on the Alternating Proximal Gradient Descent (APGD) framework, we develop Receding Horizon Alternating Proximal Descent (RHAPD) for proximable, non-smooth and strongly convex stage costs, and RHAPD-Smooth (RHAPD-S) for non-proximable, smooth and strongly convex stage costs. In addition to outperforming gradient descent-based algorithms, while maintaining a comparable runtime complexity, our proposed algorithms also allow us to solve a wider range of problems. We provide theoretical upper bounds on the dynamic regret achieved by the proposed algorithms, which decay exponentially with the length of the lookahead window. The performance of the presented algorithms is empirically demonstrated via numerical experiments on non-smooth regression, dynamic trajectory tracking, and {economic power dispatch} problems.
\end{abstract}
 

\section{Introduction}
Most sequential decision-making problems in control, robotics, and machine learning fall under an abstract unifying framework: the system needs to make a decision at every time step, after which it pays a cost or gains a reward, and the goal is to minimize the cumulative cost incurred over all time-steps. Mathematically, such problems can be formulated within the online convex optimization (OCO) framework, where we seek to minimize $\sn f_t(x_t)$,  where $x_t \in \Rn^d$ is the action taken at time $t$ and $f_t(x_t)$ is the cost incurred. In many applications, the trajectory $\{x_t\}_{t=1}^N$ of the system is generally required to be \emph{smooth} with {minimal} changes in $x_t$ over consecutive time steps. For instance, in the context of autonomous navigation, the trajectory of a vehicle must be free of jerks, and likewise, in video streaming, rapidly fluctuating bit rates result in poor quality of experience for the user. The smoothed OCO (SOCO) framework allows for smooth trajectories by seeking to minimize $\sn f_t(x_t) + g(x_t,x_{t-1})$, where the switching cost $g$ is employed to discourage large changes in $x_t$. SOCO problems arise in a number of areas, such as data center management \cite{lin2012dynamic, lin2012online}, autonomous navigation \cite{rios2016survey}, network resource allocation \cite{chen2017online}, smart grid control \cite{badiei2015online, kim2016online}, thermal management of systems-on-chip \cite{zanini2010}, and portfolio management \cite{hazan2007logarithmic, mcmahan2012no}.

Classical SOCO approaches adopt a model, wherein the stage cost $f_t$ is revealed, based on which an action \(x_{t}\) is taken. In real-world systems, however, $f_t$ depends on the environment and may be predictable to a certain extent. Indeed, predictions of $f_t$ for the near future may be available from an independent analysis of the past interactions of the system with its environment and may be quite accurate. For instance, in autonomous navigation tasks, the environmental states may be predictable over a finite horizon. In the context of smart grid control, the electricity prices for the next day may be exactly known from the day-ahead auction results. Hence, in this work, we depart from the classical SOCO setting by assuming that $\{f_t, \ldots, f_{t+W-1}\}$ are revealed at time $t$. We call this problem \emph{SOCO with predictions}. The SOCO with predictions problem can generally be tackled within the Model Predictive Control (MPC) framework, where a $W$-sized optimization problem {is formulated and solved} at each $t$. {Though empirically efficient, solving this optimization problem at each iteration can be computationally expensive. } 

An online algorithm for solving the SOCO (with predictions) problem was proposed in \cite{lin2012dynamic} in the context of data center sizing, and was shown to attain a cost that was at most 3 times the minimum attainable cost. Subsequent works such as \cite{chen2015online, chen2016using}  considered the more general case of noisy predictions, but still required solving an optimization problem at each $t$ and hence {would not be} applicable to large-scale settings. More recently, the computational bottleneck of \cite{lin2012dynamic, chen2015online,chen2016using} has been partly resolved by \cite{li2020online}, {which puts forth} first-order approaches to solve the desired problem for the quadratic switching cost. {Namely,} the {Receding Horizon Gradient Descent (RHGD) algorithm and its accelerated variant (RHAG)} proposed in \cite{li2020online} are applicable to large-scale settings and come with tight dynamic regret bounds. {However, since} these algorithms can be viewed as variants of the offline gradient descent algorithm, they require the stage costs to be smooth. It is remarked that non-differentiable cost functions are not uncommon, e.g., in multi-task lasso regression and non-smooth barrier functions in autonomous driving. 

\begin{table*}[t]
\centering
\begin{tabular}{|c|c|c|c|}
\hline
Algorithm (\(\mathcal{A}\)) & \(f_{t}\) & \(g(x_{t}, x_{t - 1})\) & \(\mathrm{Reg(\mathcal{A})}\) \\ \hline
RHGD \cite{li2020online, li2018using} & smooth + strongly convex & quadratic & \(\mathcal{O}\big(c_{1} ^ {W}(\sum_{t = 1}^{N} \norm{\theta_t - \theta_{t - 1}})\big),\) \(c_{1} \in (0, 1)\) \\ \hline
RHAG \cite{li2020online, li2018using} & smooth + strongly convex & quadratic & \(\mathcal{O}\big(c_{2} ^ {W}(\sum_{t = 1}^{N} \norm{\theta_t - \theta_{t - 1}})\big),\) \(c_{2} \in (0, 1)\) \\ \hline
RHAPD (This work) & strongly convex & assumption \ref{g_convex} & \(\mathcal{O}\big(c_{3} ^ {W}(\sum_{t = 1}^{N} \norm{\theta_t - \theta_{t - 1}})\big),\) \(c_{3} \in (0, 1)\) \\ \hline
RHAPD-S (This work) & smooth + strongly convex & quadratic & \(\mathcal{O}\big(c_{4} ^ {W}(\sum_{t = 1}^{N} \norm{\theta_t - \theta_{t - 1}})\big),\) \(c_{4} \in (0, 1)\) \\ \hline
\end{tabular}
\caption{{Comparison of the proposed algorithms with the existing algorithms for the considered problem (SOCO with predictions). The dynamic regret $\mathrm{Reg(\mathcal{A})}$ is defined in  \eqref{regret}. It is already known from 
\cite{li2020online, li2018using} that $c_2 < c_1$. Remark \ref{bound_comparsion} compares $c_4$ with $c_1$ and $c_2$.}}
\label{tab:comparison_prior_works}
\end{table*}

\subsection{Contributions}
This work builds upon \cite{li2020online} and puts forth low-complexity receding horizon proximal descent-based algorithms for smooth {as well as} non-smooth stage costs. For non-smooth but proximable stage costs, we propose Receding Horizon Alternating Proximal Descent (RHAPD). Different from existing works, {we provide a dynamic regret bound for RHAPD that holds for any} generic convex, and smooth switching cost. For non-proximable but smooth stage costs, we propose {RHAPD-Smooth} (RHAPD-S) and similarly bound its dynamic regret for quadratic switching costs. As in \cite{li2020online}, the developed bounds are proportional to the path length and exponentially decaying in the size of the lookahead window $W$. A key novelty in both the proposed algorithms is the use of the most recently updated iterates while calculating the updates. As a result, the proposed algorithms are no longer  variants of the RHGD algorithm from \cite{li2020online} and instead incorporate a flavor of \emph{alternating} descent methods. Indeed, we show that the classical alternating minimization or block-coordinate descent algorithms are special cases of RHAPD. These modifications are shown to yield significant performance gains over both RHGD {as well as} RHAG, even though the proposed algorithms are not accelerated. We demonstrate the superiority of the proposed algorithms through various numerical experiments in {economic power dispatch,} non-smooth regression, and trajectory tracking.

\section{Related Work}
In this section, we discuss some related work in the areas of online {convex} optimization, and alternating minimization. 

\subsection{Dynamic Regret} OCO problems have been widely studied for online learning, where the focus has been on characterizing their static regret, which is defined as the difference between the cumulative cost incurred by an online algorithm and that of the best-fixed decision in hindsight \cite{hazan2007logarithmic}. In dynamic and non-stationary environments, however, the optimal decision may also be time-varying, motivating the need for other more general metrics. The \emph{dynamic regret} compares the cumulative cost incurred by the algorithm against that achieved by a dynamic sequence of comparators $\{u_1, \ldots, u_N\}$ \cite{zinkevich2003online}. {{Dynamic regret is generally not sublinear but is bounded by terms that depend on the variability of the comparators, such as the path length} $\pn:= \sum_{t=2}^N \norm{\theta_t - \theta_{t-1}}$, $ \theta_t:= \argmin f_{t}(x) $ among others \cite{mokhtari2016online}.} {The competitive ratio, defined as the ratio of the cumulative cost attained by the online algorithm and that of the optimal offline algorithm is another common metric used to evaluate the performance of online algorithms}  \cite{lin2012dynamic, lin2012online}. 

\subsection{{Smoothed Online Convex Optimization}}
{The SOCO problem was first introduced in \cite{lin2012dynamic} for dynamic data center scheduling, {which proposed an online algorithm with a competitive ratio of 3 for the SOCO problem with the following switching cost: $g(x_{t}, x_{t - 1}) = \max (x_t - x_{t - 1}, 0)$}. {For \(d = 1\), \cite{bansal2015} proposed a 2-competitive randomized algorithm for the SOCO problem with the linear switching cost \(g(x_{t}, x_{t - 1}) = \abs{x_{t} - x_{t - 1}}\), matching the lower bound of \(2\) \cite{antoniadis2017tight}}. {In general, however, the competitive ratio for the SOCO problem with non-negative convex stage costs and linear switching cost \(g(x_{t}, x_{t - 1}) = \norm{x_{t} - x_{t - 1}}\) is lower bounded by $\Omega(\sqrt{d})$ \cite{chen2018}, and that for the quadratic switching cost \(g(x_{t}, x_{t - 1}) = \frac{1}{2}\norm{x_{t} - x_{t - 1}} ^ 2\) is unbounded \cite{goel2019beyond}, i.e., there exists a sequence of convex functions such that the competitive ratio of any online algorithm for solving the SOCO problem with the quadratic switching cost is \(\infty\). The \(\Omega(\sqrt{d})\) lower bound for the linear switching cost problem suggests that there cannot exist an algorithm for solving the SOCO problem, that achieves a dimension-free constant competitive ratio for general convex costs.} Notwithstanding the negative result, the Online Balanced Descent (OBD) algorithm proposed in \cite{chen2018} achieves a competitive ratio of $3 + \O(\frac{1}{\alpha})$ for $\alpha$-locally polyhedral stage costs {with the linear switching cost} and  $3 + \O(\frac{1}{\mu})$ for $\mu$-strongly convex stage costs \cite{goel2019online} {with the quadratic switching cost}. The greedy and regularized variants of OBD have subsequently been proposed in \cite{goel2019beyond}, achieving {order-wise (with respect to \(\mu\)) optimal competitive ratios for the SOCO problem with the quadratic switching cost}. {More recently, the constrained online balanced descent (COBD) algorithm, proposed by \cite{argue2020dimension} attains a competitive ratio \(\mathcal{O}(\big(\frac{L}{\mu}\big) ^ {\frac{1}{2}})\) for the SOCO problem with \(L\)-smooth and \(\mu\)-strongly convex costs and linear switching cost.}}

{From \cite{daniely2019competitive}, it is known} that for convex stage costs, no algorithm can simultaneously be constantly competitive and have sublinear regret. {{Contrarily, in} the strongly convex stage cost setting, the regularized OBD algorithm \cite{goel2019beyond} {attains} a constant competitive ratio and has a sublinear regret when considering quadratic switching costs}.  Finally, for SOCO with predictions, a lower bound on the dynamic regret was proposed in \cite{li2020online} for smooth, strongly convex stage costs and quadratic {switching cost}. Interestingly, the dynamic regret of the RHAG algorithm proposed in \cite{li2020online} {matches the lower bound order-wise with respect to the path length \(\mathcal{P}_{N}\)}. The dynamic regret bound of the RHAPD-S algorithm proposed in the present work is better than RHGD \cite{li2020online} but worse than RHAG \cite{li2020online}, {however}, the empirical performance of RHAPD-S is comparable to RHAG, while {RHAPD is better than RHAG} for proximable and smooth stage costs.

\subsection{Alternating Proximal Gradient Descent (APGD)} The APGD algorithm has been extensively studied for solving problems of the form $\min_{x_i} \sum_i F_i(x_i) + H(x_1, \ldots, x_N)$, especially for the non-convex setting \cite{nikolova2019alternating, xu2013block}. The key idea of APGD is to sequentially update the coordinates \(x_{1}, \ldots, x_{N}\) using a proximal step. {Unlike} proximal gradient methods, APGD evaluates each update using the most recently updated values of all other coordinates. The special case of two-component minimization $\min_{x, y} F(x) + G(y) + H(x,y)$ was studied in \cite{attouch2010proximal,bolte2014proximal}, while the general problem was studied in \cite{nikolova2019alternating, xu2013block}, under the Kurdyka-Lojaseiwicz (KL) assumption on the objective function \cite{lojasiewicz1963propriete,li2018calculus,bolte2010characterizations}. While the proposed RHAPD and RHAPD-S algorithms also use the most recent updates at each iteration (hence the name \emph{alternating}), our focus is on developing online algorithms and characterizing their dynamic regret. The subgradient bounds developed in the present work build upon similar results developed in \cite{attouch2013convergence,bolte2017error}, {which study subgradient sequences of functions satisfying the KL property}.

Finally, we remark that the SOCO problem can be modeled as a regularized block multi-convex optimization problem \cite{nikolova2019alternating}. Proximal variants of the block-coordinate descent (BCD) algorithm have been studied in \cite{bolte2014proximal,xu2013block}, and works such as \cite{attouch2013convergence,beck2016alternating,hesse2015,pock2016inertial} have examined these algorithms under more relaxed assumptions. 

\textbf{Notation}: We use both bold and regular font letters to denote vectors, with the bold-faced letters denoting super-vectors that collect multiple vectors. For example, we have $x_t \in \Rn^{d}$, whereas the super-vector $\x\in \Rn^{dN}$ collects $\{x_t\}_{t=1}^{N}$. The symbols \(\mathbf{0}, \mathbf{1}, \I\) represent the vector of all zeroes, all ones, and the identity matrix respectively, with the sizes being inferred from the context. The $\ell_2$ norm  of $x$ is denoted by $\norm{x}$ and \(\ip{u}{v}\) denotes the inner product of \(u\) and \(v\), defined as \(\ip{u}{v} := u ^ {\T}v\).  The gradient of the function $f(x)$ is denoted by $\nabla f(x)$, while the gradients of a bivariate function $g(x,y)$ with respect to its first and second arguments are denoted by $\nabla_1g(x,y)$ and $\nabla_2g(x,y)$, respectively. When the function is non-differentiable, we use $\partial$ in place of $\nabla$ to denote the corresponding sub-gradient. $\mathcal{P}_{\mathcal{X}}(.)$ denotes projection onto the set $\mathcal{X}$, i.e., \(\mathcal{P}_{\mathcal{X}}(x) := \argmin_{y \in \mathcal{X}} \norm{x - y}\). Finally, \(\ind_{\mathcal{X}}\) denotes the indicator function, defined as \(\ind_{\mathcal{X}}(x) = 0\) when \(x \in \mathcal{X}\) and \(\infty\) otherwise.

\section{Problem Formulation}\label{problem_setting}
We consider the following optimization problem with time-varying stage costs $f_t$ and a switching cost function $g$:
\begin{align}\label{eq:problem}
	\min_{\{x_t \in \cX\}_{t=1}^{N}} J(\x) :=  \sum_{t = 1}^{N} f_t(x_t) + g(x_t, x_{t - 1}),\tag{$\mathcal{P}$}
\end{align}
where $\cX \subseteq \Rn^d$ is a non-empty, closed, and convex set, $x_0 \in \cX$ is given, and $\x \in \Rn^{dN}$ collects the optimization variables $\{x_t\}_{t=1}^N$. For the sake of brevity, we denote $\bX:=\times_{t=1}^{N} \cX$ so that $\x \in \bX$. The stage costs are non-negative functions, for all $1\leq t \leq N$ and the switching cost is non-negative, i.e., $g(x,y) \geq 0$ for all $x,y \in \cX$. {An example formulation for trajectory planning would represent $x_t$ as the position of a robot at time $t$, the stage costs $f_t$ to model the cost of being away from the target position, and an appropriate switching cost to discourage sudden changes in {the} position or velocity of the robot.} The current work will allow for generic switching costs, though a particular case of the switching cost $g(x_t,x_{t-1}) = \frac{\gamma}{2}\norm{x_t-x_{t-1}}^2$ for $\gamma > 0$ will be considered in Section \ref{sec-prox}\footnote{In Section \ref{sec-prox}, we propose and analyze RHAPD for general switching costs, {and additionally}, show that the bounds can be significantly tightened for the quadratic switching cost. However, in Section \ref{smoothOCO} we propose and analyze RHAPD-S {just} for the quadratic switching cost.} and Section \ref{smoothOCO}. Typically, the feasible set $\cX$ depends on the intrinsic physical limitations of the system, such as the maximum speed, and is hence known in advance. Likewise, the form of the switching cost function $g$ is decided at the formulation stage and is again assumed known. In general, however, the stage costs $f_t$ may depend on extrinsic factors such as the environment or the target position, and are not known a priori, e.g., if the environment is dynamic and/or the goal is evasive. 

We consider the finite horizon look-ahead setting from \cite{li2020online}, wherein at time $t$, the agent receives stage costs $f_t$, $f_{t+1}$, $\ldots$, $f_{t+W-1}$ for the next $W$ stages and subsequently takes an action $x_t$. The agent suffers the stage cost $f_t(x_t)$ as well as the switching cost $g(x_t, x_{t-1})$, and seeks to minimize $J(\x)$ in \eqref{eq:problem}. {In contrast to the classical OCO framework (where $f_t(x_t)$ is revealed only after $x_t$ is chosen), this framework reveals the entire functional form of $f_t$, $f_{t+1}$, $\ldots$, $f_{t+W-1}$, before $x_t$ is chosen\footnote{Due to the sequential nature of the problem, the agent has full information about \(f_{1}, \dots, f_{t + W - 1}\) at time \(t\), based on which it takes the action \(x_{t}\).}}. {Such a setting is applicable to} real-world control problems where both environmental and target dynamics may be predictable, albeit over a small window of duration $W$. 

The performance of an algorithm $\cA$ is measured by its dynamic regret -- the difference between the objective value achieved by $\cA$ and the optimal value of \eqref{eq:problem}, given by
\begin{align}
	\Reg(\cA) = J(\x^\cA) - J(\x^\star),\label{regret}
\end{align}
where $\x^\cA$ collects the actions taken by $\cA$ and {$J(\x^\star)$} denotes the optimal objective value of \eqref{eq:problem}. In general, dynamic regret may not necessarily be sublinear and instead expressed in terms of a regularity measure that captures the inherent variability of the environment. One such commonly used regularity measure is the path length, which captures the variations in the stage-cost minimizers and is given by 
\begin{align}
\pn = \sum_{t=2}^N \norm{\theta_t-\theta_{t-1}},	\label{path}
\end{align}
with $\theta_t = \argmin_{x \in \cX} {f_t(x)}$ for $1 \leq t \leq N$.

It is worth mentioning that \eqref{eq:problem} {can also be expressed as}
\begin{align}
	\min_{\x} \Jt(\x) &:= \sum_{t = 1}^{N} (f_t(x_t) + g(x_t, x_{t - 1})) + \sum_{t = 1}^{N} \ind_{\cX}(x_t), \nonumber \\ &= J(\x) +  \sum_{t = 1}^{N} \ind_{\cX}(x_t). \label{jtdef}
\end{align}
For the sake of brevity, we will denote 
\begin{align}\label{FHdef}
	F(\x) &= \sn f_t(x_t), \quad H(\x) =  \sum_{t = 1}^{N} g(x_t, x_{t - 1}),
\end{align}
so that the objective can be written as $	J(\x) =  F(\x) + H(\x)$.

Before developing the algorithms to solve \eqref{eq:problem} and analyzing their performance, we state the necessary assumptions on the structure of the problem. We begin with the following subgradient boundedness assumption which is common to all the results developed in this work.
\begin{assumption}\label{bounded_subgradient}
	The stage costs $f_t:\Rn^d\rightarrow \Rn ^ +$ are $G$-Lipschitz over \(\cX\) for all $1 \leq t \leq N$, so that $\norm{\partial f_t(x)} \leq G$ for all $x \in \cX$. 
\end{assumption}
Assumption \ref{bounded_subgradient} is utilized in one of the final steps necessary to bound the optimality gap in terms of $\pn$ \eqref{path}, {hence} required to obtain the dynamic regret. 

Section \ref{sec-prox} considers proximable stage costs, and hence relies on the following key assumption: 

\begin{assumption}\label{non_smooth}
	The stage cost functions $f_t:\Rn^d\rightarrow \Rn_{+}$ are $\mu_t$-strongly convex over $\cX$ for all $1\leq t \leq N$. Further, there exists a positive constant $\mu \leq \min_{1\leq t \leq N}\mu_t$ that does not depend on \(N\). 
\end{assumption}

{Assumption \ref{non_smooth} implies {the following} quadratic lower bound on $f_t$:
\begin{align}\nn
	f_t(y) \geq f_t(x) + \ip{\partial f_t(x)}{y - x} + \frac{\mu_t}{2}\norm{y - x}^{2},
\end{align}
for all $x$, $y \in \cX$. }

The results in Section \ref{sec-prox}{,} developed for general convex switching costs, {require the following assumption}:
\begin{assumption}\label{g_convex}
	The switching cost function $g:\Rn^d \times \Rn^d \rightarrow \Rn_{+}$ is convex and $l_g$-smooth over \(\cX\), and satisfies
	\begin{align}
		0 \leq g(x,y) \leq \frac{\gamma}{2}\norm{x-y}^2, \label{gbound}
	\end{align}
 {where $\gamma > 0$ is a constant.} 
\end{assumption}
Assumption \ref{g_convex} implies that
\begin{align}\nn
	g(x, y) \geq g(u, v) & + \ip{\nabla_1g(u,v)}{x-u} + \ip{\nabla_2 g(u,v)}{y-v},
\end{align}
for all $x$, $y$, $u$, $v \in \cX$.
Further, the $l_g$-smoothness of $g$  implies
\begin{align}
	\norm{\begin{bmatrix} \nabla_1 g(u, v) \\ \nabla_2 g(u, v)\end{bmatrix} - \begin{bmatrix} \nabla_1 g(x, y) \\ \nabla_2 g(x, y)\end{bmatrix}} &\leq l_{g} \norm{\begin{bmatrix} u \\ v \end{bmatrix} - \begin{bmatrix} x \\ y\end{bmatrix}}, \label{g_smooth}
\end{align}
for all \(x, y, u, v \in \mathcal{X}\). Note that \eqref{g_smooth} implies the following partial smoothness conditions: \begin{align}
	\norm{\nabla_1g(u,y)-\nabla_1g(x,y)}&\leq l_g\norm{u-x}, \label{g_smoothx}\\
	\norm{\nabla_{{2}}g(x,v)-\nabla_{{2}}g(x,y)}&\leq l_g\norm{v-y}, \label{g_smoothy}
\end{align}
for all $x$, $y$, $u$, $v \in \cX$. 

Section \ref{smoothOCO} considers non-proximable stage costs and quadratic switching costs, and requires the following assumption:
\begin{assumption}\label{smooth}
	The stage cost functions $f_t:\Rn^d\rightarrow \Rn_{+}$ are $\mu_t$-strongly convex and $l_t$-smooth over $\cX$ for all $1\leq t \leq N$. Further, there exist positive constants $\mu \leq \min_{1\leq t \leq N}\mu_t$ and $l \geq \max_{1 \leq t \leq N} l_t$ that do not depend on $N$. 
\end{assumption}
Assumption \ref{smooth} implies the following quadratic bounds:
\begin{align*}
	\frac{\mu_t}{2}\norm{y - x}^{2} &\leq  f_t(y) - f_t(x) - \ip{\nabla f_t(x)}{y - x} \leq  \frac{l_t}{2}\norm{y - x}^{2},
\end{align*}
which hold for all $x$, $y \in \cX$. 

{Our metric of consideration is dynamic regret, which has been studied in several works on SOCO with strongly convex costs} \cite{li2020online,li2019online,dixit2019online,goel2019beyond,vaze2022dynamic}. In particular, \cite{dixit2019online} {describes} the non-smooth strongly convex setting, {which we consider in Section \ref{sec-prox}}. The Lipschitzness assumption \ref{bounded_subgradient} is a {standard} assumption when the set \(\mathcal{X}\) is constrained \cite{mokhtari2016online, li2019online}. {While developing algorithms that utilize gradient information when \(f_{t}\) is differentiable,} smoothness is a common assumption \cite{li2020online,li2019online}. {Regarding the switching cost function \(g\), \cite{lin2012dynamic,lin2012online} propose competitive algorithms for the SOCO problem with the function} \(g(x, y) = \max(x - y, 0)\) which is convex but not smooth. {Since {the algorithms we devise in this paper} also utilize the gradient information of \(g\), smoothness of \(g\) is again a natural assumption.}

\section{RHAPD for Proximable Stage Costs}\label{sec-prox}

{{This section considers the case of non-smooth strongly convex stage costs. We first discuss the classical proximal gradient descent method and then modify it to yield an improved APGD algorithm. Subsequently, we provide an initialization policy as well as an online implementation of APGD, resulting in the proposed RHAPD algorithm. Finally, we also show that the classical alternating minimization algorithm can be viewed as a special case of the RHAPD algorithm.}}

\begin{algorithm*}[h]
	\caption {\textbf{R}eceding \textbf{H}orizon \textbf{A}lternating \textbf{P}roximal \textbf{D}escent (RHAPD)}
	\begin{algorithmic}[1]
		\STATE\textbf{Input:} $x_{0}, \cX, \gamma > 0, W, N;$
		\STATE \textbf{Initialize} \(x_1 \^ 0 = x_{0};\)
		\STATE \textbf{for} \(t = 2 - W\) \textbf{to} \(N\)
		\STATE \hspace{3mm}\textsc{Step 1: }Initialize \(x_{t + W} \^ 0 \in \cX\) by an initialization policy \(\mathcal{I}\) given as \(x_{t + W} ^ {(0)} = \argmin_{x\in \mathcal{X}} f_{t + W - 1}(x);\)
		\STATE \hspace{3mm}\textsc{Step 2: } \textbf{Update} \(x_{t + W - 1}, x_{t + W - 2}, \ldots, x_t\)
		\STATE \hspace{3mm}\textbf{for} \(i = t + W - 1\) \textbf{down to} \(t\)
			\STATE \hspace{3mm}\hspace{3mm} \(k = t + W - i;\)
			\STATE \hspace{3mm}\hspace{3mm} Update \(x_i\^k\)(using most recent updated values) in the following manner:			\begin{align}\nn
   x_i\^k & = \prox_{\tau_if_i + \ind_{\cX}}(x_i\^{k-1} - \tau_i \nabla H_i\^k(x_i\^{k-1})), \\
				& = \left. \begin{cases} 
                    \prox_{\tau_if_i + \ind_{\cX}} \big(x_i\^{k-1} - \tau_i(\nabla_1 g(x_i\^{k-1}, x_{i - 1}\^k) + \nabla_{2} g (x_{i + 1} \^ {k-1}, x_i\^{k-1}))\big) & \text{if\,\,} 1 \leq i \leq N - 1, \\
                    \prox_{\tau_i f_i + \ind_{\cX}}\big(x_i\^{k-1} - \tau_i\nabla_1g(x_i\^{k-1}, x_{i - 1}\^k)\big) & \text{if\,\,} i = N.
                        \end{cases}\right. \nn
		\end{align}
			\STATE\hspace{3mm}\textbf{end}
		\STATE \textbf{Output: } \(x_t \^ W\) at stage \(t \geq 1;\)
		\STATE\textbf{end}
	\end{algorithmic}
	\label{RHAPD}
\end{algorithm*}

\subsection{Alternating Proximal Gradient Descent (APGD) Algorithm}\label{apgmns}
Towards developing the proposed algorithm, we first consider solving \eqref{eq:problem} using proximal gradient descent (PGD), which is a classical first-order method {with the following update:}
\begin{align}
	\x\^k = \prox_{\tau F + \ind_{\bX}}\left(\x\^{k-1} - \tau \nabla H(\x\^{k-1})\right), \label{offline_proximal}
\end{align}
where $\x\^k$ collects the time-indexed iterates $\{x_t\^k\}_{t=1}^{N}$ and $\tau > 0$ is a suitably chosen step size. {Although superior to the subgradient method in terms of iteration complexity, this method is only practical when the {$\prox$} operation}
\begin{align}	
	\prox_{{\psi}}(y) &:= \arg\min_{x} {\psi(x)} + \frac{1}{2}\norm{x-y}^{2},  \label{proxdef}
\end{align}
 is {easily computable} for a given function $ {\psi} $. {For our problem,}
 \begin{align}
	\x\^k &\eqtext{\eqref{proxdef}} \argmin_{\x \in \bX} F(\x) + \frac{1}{2\tau} \norm{\x  - \x\^{k-1} + \tau \nabla H(\x\^{k-1})}^{2}, \nonumber\\
	&\eqtext{\eqref{FHdef}} \argmin_{\{x_t \in \cX\}_{t=1}^{N}} \sum_{t = 1}^{N} \bigg(f_t(x_t) + \frac{1}{2\tau}\norm{x_t - x_t\^{k-1} + \tau [\nabla H(\x\^{k-1})]_t}^{2}\bigg), \label{separable} 
\end{align}where
\begin{align*}
	[\nabla H(\x\^{k-1})]_t = \left.
	\begin{cases}
		\nabla_1 g(x_t\^{k-1}, x_{t - 1}\^{k-1}) + \nabla_{2} g(x_{t + 1}\^{k-1}, x_t\^{k-1}) & t < N, \\
		\nabla_1 g(x_t\^{k-1}, x_{t - 1}\^{k-1}) & t = N.
	\end{cases}
	\right.
\end{align*}
Since the $t$-th summand in \eqref{separable} depends only on $x_t$, we can express \(x_{t} ^ {(k)}\) as
\begin{align} \label{prox_update}
			    x_t\^k  &= \prox_{\tau f_t + \ind_{\cX}}(x_t\^{k-1} - \tau [\nabla H(\x\^{k-1})]_t), 
\end{align}
for all $1\leq t \leq N$. The performance of PGD has been well-studied, and under \ref{g_convex}, we can similarly write down the accelerated PGD \cite{beck2009fast} {or FISTA, as it is generally known}. 

Observe that the update for \(x_t\^k\) in \eqref{prox_update} depends on $x_t\^{k-1}$, $x_{t - 1}\^{k-1}$, and $x_{t + 1}\^{k-1}$ since the gradient of \(H\) is computed at the previous iterate \(\x\^{k-1}\). However, if the updates for $x_1\^k$, $x_2\^k$, $\ldots$, $x_N\^k$ are carried out sequentially, it can be seen that {although the iterate $x_{t-1}\^k$ is already available at the time of updating $x_t\^k$, \eqref{prox_update} is using its previous value $x_{t-1}\^{k-1}$}. This observation motivates us to consider the \emph{alternating} proximal gradient descent method, {for which the} updates take the form
\begin{align}\label{alt_prox_update} 
	x_t\^k &= \prox_{\tau_tf_t+ \ind_{\cX}}(x_t\^{k-1} - \tau_t \nabla H_t\^k(x_t\^{k-1})), 
\end{align}
where
\begin{align}\label{htk}
	H_t\^k(x_t) := \begin{cases} 
			g(x_t, x_{t - 1}\^k) + g(x_{t + 1}\^{k-1}, x_t) & t < N, \\
		g(x_t, x_{t - 1}\^k) & t = N,  \end{cases}
\end{align}
where $x_0^{(k)} := x_0$ for any \(k \geq 0\). {Notice that} the update in \eqref{alt_prox_update} {uses} a suitably re-defined function $H_t\^k$ which depends on $x_{t-1}\^k$ {and not} $x_{t-1}\^{k-1}$. The resulting $\nabla H_t\^k(x_t\^{k-1})$ in \eqref{htk} {is clearly not the gradient of $H$ at any point}, and hence \eqref{alt_prox_update} is not an instance of the classical proximal gradient method. As shall be shown later, {this} seemingly small change results in a significant performance improvement.

\subsection{Receding Horizon Alternating Proximal Descent Algorithm}
We next detail the online implementation and an initialization policy for the proposed alternating proximal gradient descent algorithm. Observe first that the sequential nature of the updates \eqref{alt_prox_update} imply that at the $k$-th iteration, the blocks of $\x$ evolve as:
\begin{align}
	\x\^{k-1} = (x_1\^{k-1}, \ldots, x_N\^{k-1}) \to \dots  (x_1\^k, \ldots, x_i\^k, x_{i + 1}\^{k-1}, \ldots, x_N\^{k-1}) \to \ldots(x_1\^k, \ldots, x_N\^k) = \x\^k. \label{evolution}
\end{align}
{Figure \ref{oolemma}} depicts the sequence of updates in \eqref{alt_prox_update}. The arrows show dependencies between iterates, since for $1 \leq t \leq N-1$, the update for \(x_t\^k\) requires the availability of $x_{t - 1}\^k$, $x_t\^{k-1}$,  and $x_{t + 1}\^{k-1}$, as well as that of \(f_t\). Likewise, the update for $x_N\^k$ depends on $x_{N - 1}\^k$, $x_N\^{k-1}$, and $f_N$.

\begin{figure*}[t]
	\centering
	\begin{subfigure}[b]{0.48\textwidth}
		\includegraphics[width=\linewidth]{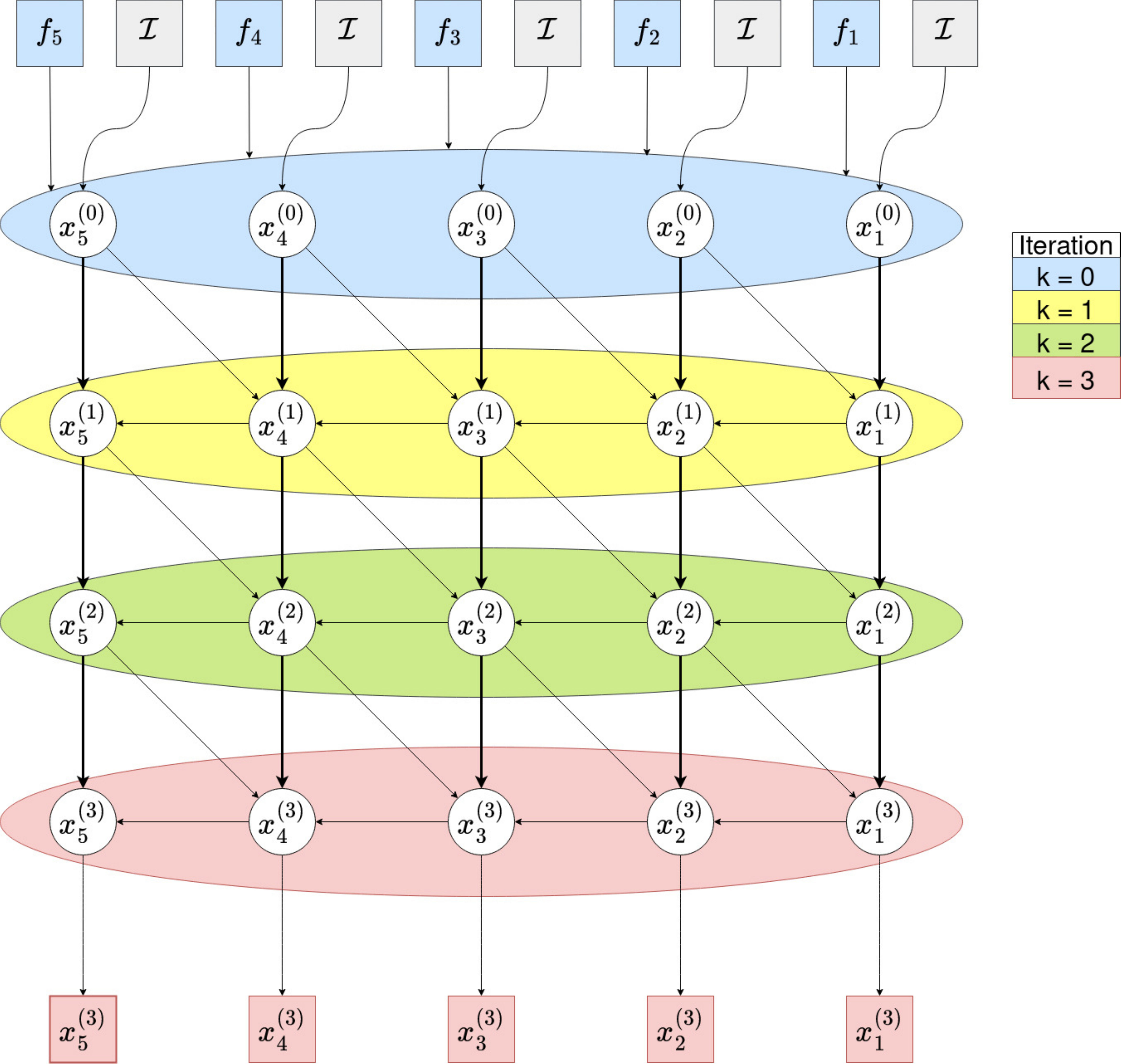}
		\caption{}
	\end{subfigure}
	\begin{subfigure}[b]{0.48\textwidth}
		\includegraphics[width=\linewidth]{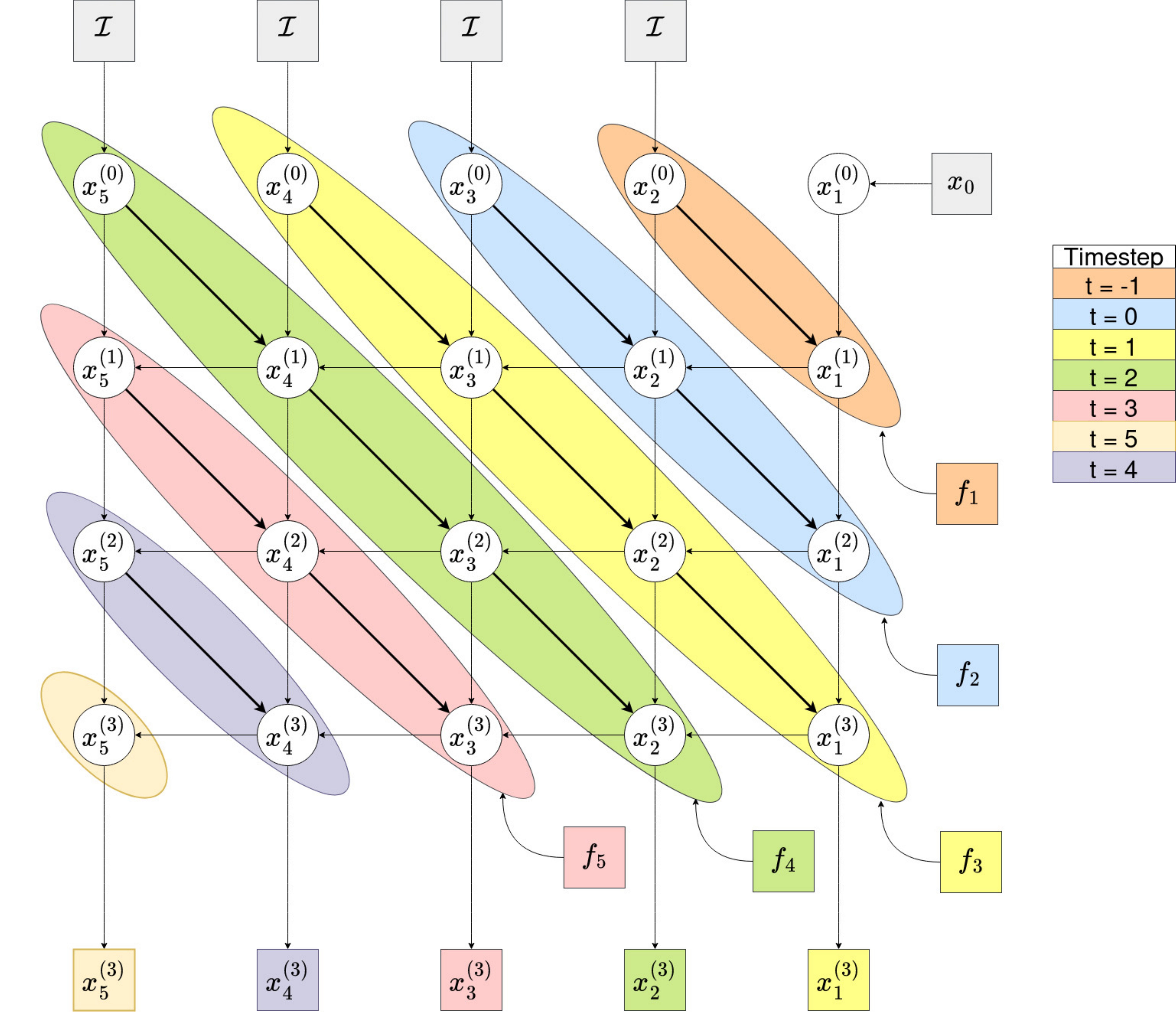}
		\caption{}
	\end{subfigure}
	\caption{Pictorial representation of the Offline-Online Lemma. Part (a) represents the offline updates, where the stage cost functions are all available during initialization. Part (b) shows the proposed online updates, where the functions are revealed sequentially. For this illustration, \(N = 5\) and \(W = 3\) is considered.}
	\label{oolemma}
\end{figure*}

The online implementation seeks to exploit this limited dependence structure to its advantage.  When $f_t$ is first revealed at time $t - W + 1 \geq 1$, it is used to calculate $x_t\^1$, which depends on $x_{t-1}\^1$, $x_t\^0$, and $x_{t+1}\^0$. Of these, $x_{t-1}\^1$ should already be available from time $t-W$ when $f_{t-1}$ was first revealed, while $x_t\^0$ and $x_{t+1}\^0$ are known from the algorithm initialization. Subsequently, $x_{t-1}\^2$ may be calculated as it depends on the previously revealed $f_{t-1}$, as well as the already available iterates $x_{t-2}\^2$,  $x_t\^1$, and $x_{t-1}\^1$, where $x_{t-2}\^2$ was similarly calculated in the previous iteration. Continuing in the same way, it can be seen that the updates for $x_{t-2}\^{3}$, $\ldots$, $x_{t-W+1}\^W$ can also be carried out.  The last of these updates, namely that for $x_{t-W+1}\^W$ is then used to take the control action at the current time $t-W+1$, and no further updates for $x_{t-W+1}$ are required. The idea is depicted in Figure \ref{oolemma}(b) for $N = 5$ and $W = 3$, where the iterates updated at the same time instant are all highlighted using the same color. 

We also note that when $f_t$ is first revealed, the update for $x_t\^1$ requires the initial value $x_{t+1}\^0$. Instead of initializing arbitrarily, we set $x_{t+1}\^0 = \argmin_{x \in \cX} f_t(x)$ for all $t \geq 1$, which should be easy to calculate since $f_t$ is proximable over \(\mathcal{X}\). For the first time instant, we initialize $x_1\^0 = x_0$. {This} initialization approach allows us to bound the initial regret with a multiple of the path length. 

The discussion so far does not apply to the boundary cases. The functions $\{f_1, f_2, \ldots, f_W\}$  are all simultaneously revealed at time $t = 1$. Therefore, before carrying out the updates for $x_1\^W$, $\ldots$, $x_W\^1$, we also need to calculate $x_1\^{W-1}$, $x_2\^{W-2}$, $\ldots$, $x_{W-1}\^1$, $x_1\^{W-2}$, $\ldots$, $x_{W-2}\^1$, $\ldots$, $x_1\^1$. These updates can be seen as corresponding to the hypothetical time instants $t = 2-W, \ldots, 0$, as depicted in Figure \ref{oolemma}(b). Also, no new functions are revealed for $t > N-W+1$, {therefore} only the updates for $x_t\^W$, $\ldots$,  $x_N\^{W+t-N}$ need to be carried out at these time instants. {{We refer to the online implementation of} \eqref{alt_prox_update} {as RHAPD and summarize it in Algorithm} \ref{RHAPD}}. The following lemma summarizes the discussion in this subsection. 

\begin{lemma}[\textbf{Offline-Online Lemma}] \label{offline_online}
	For a given initialization and step size rule, the iterates $x_t\^W$ obtained from Algorithm \ref{RHAPD} match the iterates $x_t\^W$ obtained from running \eqref{alt_prox_update} for $W$ iterations. 
\end{lemma}

\begin{rem}\label{pgd_and_fista}
	The proximal gradient descent algorithm in \eqref{offline_proximal} (and its accelerated version) can likewise be implemented in an online fashion, thanks to the separable structure of the objective in \eqref{eq:problem}. {Still,} the proposed RHAPD algorithm, which uses the most up-to-date iterates in \eqref{prox_update}, outperforms both PGD {as well as its accelerated variant, FISTA.}

\end{rem}

\subsection{Regret Bound for RHAPD}\label{sec:regret-ns}

{ Thanks to the above Lemma \ref{offline_online}, most of our analysis can be focused on the updates in \eqref{alt_prox_update} for $W$ iterations. Under Assumption \ref{non_smooth}, we consider strongly convex and proximable, but possibly non-smooth stage costs. Theorem \ref{dynamic_regret_non_smooth_general_g} provides the regret bound for the case of any general convex switching cost under Assumption \ref{g_convex} and Theorem \ref{rhapd_quadratic_regret} specializes to the quadratic switching cost. In both cases, the overall analysis requires the following four key steps: 
\begin{enumerate}[label=\roman*)]
    \item Establishing a sufficient decrease property which bounds the per-iteration decrease in the objective. Lemma \ref{decrease_non_smooth} {in Appendix \hyperref[proof_of_decrease_non_smooth]{B} establishes} that the iterates \(\{\x\^k\}\) generated by \eqref{alt_prox_update} satisfy
		\begin{align}\label{sufficient_decrease_non_smooth_theory_marked}J(\x\^k) - J(\x\^{k-1}) \leq  - \rho \norm{\x\^k - \x\^{k-1}}^{2}, \end{align}
  where \(\rho := \frac{\mu}{2} + \frac{1}{\tau}-l_g.\) Therefore, for a decrease in the objective \(J\), we can choose the step size such that \(\rho > 0\), a common choice of the step size being \(\tau = \frac{1}{l_g}.\)
    \item Establishing a bound on the subgradient of $\Jt$. In particular, in Lemma \ref{subgrad_bound_non_smooth} {in Appendix \hyperref[proof_of_subgrad_bound_non_smooth]{C} }, we establish that for all $k \geq 1$, there exists a $\v\^k \in \partial \Jt(\x\^k)$, such that 
	\begin{align}\label{subgrad_bound_theory_marked}
		\norm{\v\^k} \leq \beta \norm{\x\^k - \x\^{k-1}},
	\end{align}
	 where $\beta^2 := 2\left(\sqrt{5}l_g + \frac{1}{\tau}\right)^2$. 
    \item Using the Polyak-Lojaseiwicz (PL) inequality and \eqref{sufficient_decrease_non_smooth_theory_marked}, \eqref{subgrad_bound_theory_marked} to establish a linear decrease in the optimality gap with respect to the initial optimality gap. 
  \begin{lemma}[\textbf{Rate of Convergence of APGD}]\label{rate_APGD_Non_Smooth}
	Under  \eqref{non_smooth} and \eqref{g_convex}, the optimality gap of the APGD algorithm decays as \begin{align}\label{opt_gap_bound}
		J(\x\^k) - J(\x^\star) \leq  \left(1+\frac{2\mu \rho}{\beta^2}\right)^{-k}\left(J(\x\^{0}) - J(\x^\star)\right),
		\end{align}
	where $\rho = \frac{\mu}{2} + \frac{1}{\tau} - l_g$ and $\beta^2 = 2\left(\sqrt{5}l_g + \frac{1}{\tau}\right)^2$. 
\end{lemma}
\begin{proof}
    From assumptions \ref{non_smooth}, \ref{g_convex}  we observe that $\Jt$ is $\mu$-strongly convex. Hence, from the PL-inequality \cite{beck2017first}, we have that $\Jt(\x\^k) \leq \Jt(\x^\star) + \frac{\norm{\w}^2}{2\mu}$ for any $\w \in \partial \Jt (\x\^k)$. Here, since $\x\^k, \x^\star \in \bX$, we have that $\Jt(\x\^k) = J(\x\^k)$ and $\Jt(\x^\star) = J(\x^\star)$. Further, we choose the subgradient $\w = \v\^k$ whose norm can be bounded \eqref{subgrad_bound_theory_marked}, to obtain
	\begin{align}
		J(\x\^k)  &\leq  J(\x^\star) + \frac{\norm{\v\^k}^{2}}{2\mu}, \nonumber \\ &\stackrel{\eqref{subgrad_bound_theory_marked}}\leq  J(\x^\star)  + \frac{{\beta}^{2}}{2\mu} \norm{\x\^k - \x\^{k-1}}^{2}, \nn \\
		&\leqtext{\eqref{sufficient_decrease_non_smooth_theory_marked}}  J(\x^\star) + \frac{{\beta}^{2}}{2\mu\rho} \left(J(\x\^{k-1})-J(\x\^k)\right). \label{optgapproof}
	\end{align}
		The optimality gap can therefore be bounded by re-arranging \eqref{optgapproof} as follows
		\begin{align}
				J(\x\^k) - J(\x^\star) 
				&\leq \frac{\beta^2}{2\mu\rho}\Big(\big(J(\x\^{k-1}) - J(\x^\star)\big) -	\big(J(\x\^k) - J(\x^\star)\big)\Big),\nonumber \\
				&\leq \frac{J(\x\^{k-1}) - J(\x^\star)}{\left(1 + \frac{2\mu\rho}{\beta^2}\right)}\leq \frac{J(\x\^{0}) - J(\x^\star)}{\left(1 + \frac{2\mu\rho}{\beta^2}\right)^k}.\nn
		\end{align}
  This completes the proof.
\end{proof}
    \item Bounding the initial regret/optimality gap. {Particularly, Lemma \ref{dynamic_regret_non_smooth_initialization} in Appendix \hyperref[proof_dynamic_regret_non_smooth_general_g]{D} shows that} \begin{align}\label{reg_initialization_theory_marked}
	\mathrm{Reg}(\mathcal{I}) \leq G\left(1 + \frac{\gamma}{\mu}\right) \sum_{t = 1}^{N} \norm{\theta_t - \theta_{t - 1}}.
\end{align}
\end{enumerate}
For the sake of simplicity, we will consider a constant step size $\tau_t = \tau$ for all $1\leq t \leq N$, though the extension to the general case is straightforward.}
{
\begin{theorem}[\textbf{Regret of RHAPD for general switching cost}]\label{dynamic_regret_non_smooth_general_g}
	Under assumptions \ref{bounded_subgradient}, \ref{non_smooth}, \ref{g_convex}, the regret attained by RHAPD with the initialization policy \(\mathcal{I}\), that is \(x_{t + W} ^ {(0)} = \argmin_{x \in \mathcal{X}} f_{t + W - 1}(x)\), is bounded by \begin{equation}
		\mathrm{Reg}(\text{RHAPD}) \leq \frac{G(1 + \frac{\gamma}{\mu})}{\left(1 + \frac{2\mu \rho}{\beta ^ 2}\right) ^ {W}} \left(\sum_{t = 1} ^ {N} \norm{\theta_t - \theta_{t - 1}}\right),
	\end{equation}
	where \(\theta_0 := x_0, \theta_t = \argmin_{x \in \mathcal{X}} f_t(x)\);  $\rho = \frac{\mu}{2} + \frac{1}{\tau} - l_g$, $\beta^2 = 2\left(\sqrt{5}l_g + \frac{1}{\tau}\right)^2$, and the step size \(\tau\) is such that \(\rho > 0\).
\end{theorem}
\begin{proof}
We can bound the regret attained by RHAPD as 
    \begin{align}
	\mathrm{Reg}(\text{RHAPD}) &\stackrel{(a)}= J(\x\^W) - J(\x^\star), \nonumber \\
 &\stackrel{\eqref{opt_gap_bound}}\leq\frac{J(\x ^ {(0)}) - J(\x\^*)}{\left(1 + \frac{2\mu \rho}{\beta ^ 2}\right)^W} = \frac{\mathrm{Reg}(\mathcal{I})}{\left(1 + \frac{2\mu \rho}{\beta ^ 2}\right) ^ {W}}, \nonumber \\
 &\stackrel{\eqref{reg_initialization_theory_marked}}\leq\frac{G(1 + \frac{\gamma}{\mu})}{\left(1 + \frac{2\mu \rho}{\beta ^ 2}\right) ^ {W}} \left(\sum_{t = 1} ^ {N} \norm{\theta_t - \theta_{t - 1}}\right)\nonumber, 
\end{align}
where \((a)\) follows from Lemma \ref{offline_online}. This completes the proof.
\end{proof}
}
{
In the special case when the switching cost is quadratic \(g(x, y) = \frac{\gamma}{2}\norm{x - y} ^ 2\), we can obtain a tighter bound on the regret, given by the following Theorem.
\begin{theorem}[\textbf{Regret of RHAPD for quadratic switching cost}]\label{rhapd_quadratic_regret}
	Under assumptions \ref{bounded_subgradient}, \ref{non_smooth},  the regret attained by RHAPD with the {initialization policy \(\mathcal{I}\)}, that is \(\displaystyle x_{t + W}\^0 = \argmin_{x \in \cX} f_{t + W - 1}(x)\), is bounded by \begin{align}\nn
		\mathrm{Reg}(\text{RHAPD}) \leq \frac{G (1 + \frac{\gamma}{\mu})}{\left(1 + \frac{2\mu {\rho_q}}{\beta_q^2}\right)^{W}} \left(\sum_{t = 1}^{N} \norm{\theta_t - \theta_{t - 1}}\right),
	\end{align}
	where \(\displaystyle \theta_{0} = x_{0}, \theta_t = \argmin_{x \in \cX} f_t(x)\); $\rho_q := \frac{\mu}{2} + \frac{1}{\tau} - \gamma$, \begin{align*}
		\beta^2_q := 2\left(\gamma^2 + \max\left\{\left(2\gamma - \frac{1}{\tau}\right)^2, \left(\gamma - \frac{1}{\tau}\right)^2\right\}\right),
	\end{align*} 
 and the step size \(\tau\) is such that \(\rho_{q} > 0\).
\end{theorem}
}
{Since the quadratic switching cost is $\gamma$-smooth in both of its arguments, the bounds \eqref{sufficient_decrease_non_smooth_theory_marked}, \eqref{subgrad_bound_theory_marked} are already applicable. However, the special form of the switching cost function allows us to obtain tighter  bounds on intermediate quantities as well as the dynamic regret of Algorithm \ref{RHAPD}. {In Corollary \ref{decrease_quadratic} in Appendix \hyperref[proof_subgrad_bound_quadratic]{E}}, we show that the constants \(\rho\) and \(\beta\) in the sufficient decrease property \eqref{sufficient_decrease_non_smooth_theory_marked} and subgradient bound \eqref{subgrad_bound_theory_marked} can be tightened to \(\rho_{q}\) and \(\beta_{q}\) respectively. Finally, proceeding similarly as the proof of Theorem \ref{dynamic_regret_non_smooth_general_g}, we can prove Theorem \ref{rhapd_quadratic_regret}. The resulting steps are skipped for the sake of brevity.}
\subsection{Receding Horizon Alternating Minimization} \label{RHAM}
For the quadratic switching cost $g(x,y) = \frac{\gamma}{2}\norm{x-y}^2$ and for a particular choice of step sizes, it turns out that APGD algorithm becomes the well-known alternating minimization or block coordinate descent algorithm \cite{bezdek2003convergence,beck2015convergence, netrapalli2013phase,o2007alternating}. To see this, observe that for the quadratic switching cost, we have that $\nabla H_t\^k(x_t\^{k-1}) = \gamma (2x_t\^{k-1} - x_{t - 1}\^k - x_{t + 1}\^{k-1})$ for $1 \leq t \leq N-1$.  Letting $\tau_t = \frac{1}{2\gamma}$ for $1\leq t \leq N-1$, the updates in \eqref{alt_prox_update} can be written as 
\begin{align}
	x_t\^k &= \prox_{\frac{f_t}{2\gamma}+\ind_{\cX}} \hspace{-1mm}\Big[x_t\^{k-1} - \frac{1}{2}\left(2x_t\^{k-1} - x_{t - 1}\^k - x_{t + 1}\^{k-1}\right)\Big], \nonumber\\
 	&\hspace{-3mm}= \argmin_{x_t \in \cX} f_t(x_t)  + \frac{\gamma}{2}\norm{x_t - x_{t - 1}\^k}^2  +\frac{\gamma}{2} \norm{x_t - x_{t + 1}\^{k-1}}^2\nonumber,\\
	&\hspace{-3mm} =\argmin_{x_t \in \cX} J(x_1\^k, \ldots, x_{t - 1}\^k, x_t, x_{t + 1}\^{k-1}, \ldots, x_N\^{k-1}), \label{amt}
\end{align}
where in the last equality, we have used the notation $J(\x) = J(x_1, \ldots, x_N)$. In the same way, for $t = N$, we have that $\nabla H_N\^k(x_N\^{k-1}) = \gamma(x_N\^{k-1} - x_{N - 1}\^k)$. Setting $\tau_N = \frac{1}{\gamma}$, the update for $x_N$ in \eqref{alt_prox_update} can be written as 
\begin{align}
	x_N\^k &= \prox_{\frac{1}{\gamma}f_N(x_N) + \ind_{\cX}}(x_{N - 1}\^k) = \argmin_{x_N \in \cX}J(x_1\^k, \ldots, x_{N - 1}\^k, x_N). \label{amn}
\end{align}
To summarize, for the special case of quadratic switching costs and for a specific choice of step sizes, the proposed RHAPD algorithm reduces to applying block coordinate descent on the objective $J(\x)$. We refer to the alternating minimization algorithm in \eqref{amt}, \eqref{amn} applied to \eqref{eq:problem} as the RHAM algorithm. Naturally, the regret bounds developed in the previous section for the general RHAPD algorithm are also applicable to the RHAM algorithm.

\begin{algorithm*}[t]
	\caption {\textbf{R}eceding \textbf{H}orizon \textbf{A}lternating \textbf{P}roximal \textbf{D}escent for \textbf{S}mooth Stage Costs (RHAPD-S)}
	\begin{algorithmic}[1]
		\STATE\textbf{Input:} $x_{0}, \mathcal{X}, \gamma > 0, W, N;$
		\STATE \textbf{Initialize} \(x_1\^0 = x_{0};\)
		\STATE \textbf{for} \(t = 2 - W\) \textbf{to} \(N\)
		\STATE \hspace{3mm}\textsc{Step 1: } \textbf{Initialize} \(x_{t + W}\^0\) by OGD: 
			$x_{t + W}\^0 = \px{\cX}{x_{t + W - 1}\^0 - \eta \nabla f_{t + W - 1}(x_{t + W - 1}\^0)}$, \(\eta\) is a suitably selected step size;
		\STATE \hspace{3mm}\textsc{Step 2: } \textbf{Update} \(x_{t + W - 1}, x_{t + W - 2}, \ldots, x_t\)
		\STATE \hspace{3mm}\textbf{for} \(i = t + W - 1\) \textbf{down to} \(t\)
			\STATE \hspace{3mm}\hspace{3mm} \(k = t + W - i;\)
			\STATE \hspace{3mm}\hspace{3mm} Update \(x_i\^k\) in the following manner: \begin{align}\nn
				x_i\^k = \left.
						\begin{cases}
						  \px{\cX}{\frac{\gamma \tau_i(x_{i - 1}\^k + x_{i + 1}\^{k-1}) + x_i\^{k-1} - \tau_i \nabla f_i(x_i\^{k-1}) }{2\gamma\tau_i + 1}} & \text{if } 1 \leq i \leq N - 1, \\
						  {\px{\cX}{\frac{\gamma \tau_i (x_{i - 1}\^k) + x_i\^{k-1} - \tau_i \nabla f_i(x_i\^{k-1})}{\gamma \tau_i + 1}}} & \text{if } i = N.
						\end{cases}
				\right.
			\end{align}
			\STATE\hspace{3mm}\textbf{end}
		\STATE \textbf{Output: } \(x_t\^W\) at stage \(t \geq 1;\)
		\STATE\textbf{end}
	\end{algorithmic}
	\label{RHAPD_Smooth}
\end{algorithm*}

\section{Smooth Stage Costs and Quadratic Switching Costs}\label{smoothOCO}
This section considers the setting when the stage cost functions $f_t$ are not proximable or if the proximal operation with respect to $f_t$ is too costly. We develop algorithms that instead depend on the gradient $\nabla f_t(x)$ but require $f_t$ to be smooth. A quadratic switching cost given by $g(x,y) = \frac{\gamma}{2}\norm{x-y}^2$ is considered. As in Section \ref{sec-prox}, we first introduce the PGD algorithm and subsequently modify it to develop its alternating variant, which admits an online implementation. {This is followed by the dynamic regret bound analysis for the proposed algorithm.}

\subsection{Alternating Proximal Gradient Method for Smooth Stage Costs}
Recalling that $J(\x) = F(\x) + H(\x)$, {we can write the PGD updates differently when $F$ is not proximable in the following manner:}
\begin{align}
	\x\^k &= \prox_{{\tau} H + \ind_\bX}(\x\^{k-1} - \tau\nabla F(\x\^{k-1})), \nn \\
	&= \argmin_{\{x_t\in\cX\}_{t=1}^N} \bigg(\frac{\gamma}{2}\sn \norm{x_t-x_{t-1}}^2 + \sn\frac{1}{2{\tau}}\norm{x_t - x_t\^{k-1} + {\tau} \nabla f_t(x_t\^{k-1})}^2 \bigg),\label{coupled}
\end{align}
{The presence of terms of the form $\frac{\gamma}{2}(\norm{x_t - x_{t - 1}}^2 + \norm{x_t - x_{t + 1}}^2)$ in \eqref{coupled} does not allow separating the minimization into individual minimizations with respect to each $x_t$ as was possible in \eqref{separable}}. Thus, the proximal gradient method cannot be implemented in an online fashion for smooth stage costs. 

Towards developing an online algorithm, let us re-use our idea of sequentially updating the iterates $\{x_t\}_{t=1}^N$ while making use of the most recent updated values at each iteration. As in the APGD algorithm, we replace $H$ with 
\begin{align*}
	H_t\^k (x_t) = 
	\begin{cases}
		\frac{\gamma}{2} \bigg(\norm{x_t - x_{t - 1}\^k}^{2} + \norm{x_t - x_{t + 1}\^{k-1}}^{2}\bigg) & t < N, \\
		\frac{\gamma}{2} \norm{x_t - x_{t - 1}\^k}^{2} &  t = N,
	\end{cases}
\end{align*}
which is the same as in \eqref{htk} for the quadratic switching cost. Since $H_t\^k(x_t)$ depends only on $x_t$, it decouples the minimization in \eqref{coupled} and {allowing $x_t$ to be updated as}
 \begin{align}
 	{x_t\^k = \prox_{\tau_tH_t\^k + \ind_{\cX}}\left(x_t\^{k-1} - \tau_t \nabla f_t(x_t\^{k-1})\right)}.\label{xtk-update-smooth}
 \end{align}
We refer to the updates in \eqref{xtk-update-smooth} as the APGD-S {algorithm}. Observe that since $H_t\^k$ is quadratic, the proximal operation can be written in terms of a projection operation onto the set $\cX$. Consequently, the updates in \eqref{xtk-update-smooth} are significantly cheaper than those in \eqref{coupled}. 

\subsection{Online Implementation}\label{sec:initialization}
For the online implementation, we observe that the update for $x_t\^k$ for $1 \leq t \leq N-1$ depends on  $x_{t - 1}\^k$, $x_t\^{k-1}$, $x_{t + 1}\^{k-1}$, and $\nabla f_t(x_t\^{k-1})$, and follows the evolution depicted in \eqref{evolution}. Hence, following similar arguments as in Section \ref{sec-prox}, we see that when $f_t$ is first revealed, it can be used to calculate the updates for $x_t\^1$, $\ldots$, $x_{t-W+1}\^W$, all of which depend on available or previously updated quantities. 

As $f_t$ are not assumed proximable, the initialization strategy adopted in Algorithm \ref{RHAPD} may not be viable. Instead, we make use of the online gradient descent (OGD)-based initialization strategy, same as {in} \cite{li2020online}, where $x_t\^0 = \px{\cX}{x_{t-1}\^0 - \eta\nabla f_{t-1}(x_{t-1}\^0)}$ and $x_1\^0$ is initialized to \(x_1 ^ {(0)} = x_{0}\). The complete online algorithm is summarized in Algorithm \ref{RHAPD_Smooth}. 

\begin{rem}	As per \cite[Proposition 2]{mokhtari2016online}, choosing \(\eta \leq \frac{1}{l}\) ensures that the following gradient descent update: \(x_{t} ^ {(0)} = \px{\mathcal{X}}{x_{t - 1} ^ {(0)} - \eta \nabla f_{t - 1}(x_{t - 1} ^ {(0)})}\) satisfies the contraction property, i.e., \(||x_{t} ^ {(0)} - \theta_{t - 1}|| \leq \sqrt{1 - \frac{\mu}{L}}||x_{t - 1} ^ {(0)} - \theta_{t - 1}||\), where \(\theta_{t} = \argmin_{x \in \cX} f_{t}(x)\). In other words, the updated point \(x_{t} ^ {(0)}\) is more closer to the optimizer \(\theta_{t - 1}\) than the point \(x_{t - 1} ^ {(0)}\) from which the descent step was performed. This contraction is in fact the crucial idea behind \cite[Theorem 4]{li2020online} which bounds the initial regret attained by the OGD-based initialization strategy. Therefore, similar to \cite{mokhtari2016online} and \cite{li2020online}, we assume that \(\eta \le \frac{1}{l}\). However, we set \(\eta = \frac{1}{l}\) in Theorem \ref{regret_RHAPD-S} to get a simplified expression for the regret attained by RHAPD-S.
\end{rem}

\subsection{Regret bound for RHAPD-S}
In this section, we provide the regret bound for RHAPD-S. {The analysis proceeds along the steps laid out in Section \ref{sec:regret-ns} for bounding the regret of RHAPD. {Specifically}, we establish the sufficient decrease property {(Lemma \ref{lem:sufficient_decrease_smooth}) and subgradient bound (Lemma \ref{subgrad_bound_smooth} in Appendix \hyperref[proof_of_subgrad_bound_smooth]{G})} for the iterates produced by \eqref{xtk-update-smooth}. Further, by using the PL inequality, we can establish that the optimality gap of \eqref{xtk-update-smooth} can be bounded similarly as RHAPD. We skip this for the sake of brevity (refer Lemma \ref{rate_APGD_Non_Smooth}) instead.}

As in the proof of Theorem \ref{dynamic_regret_non_smooth_general_g}, we denote the initialization regret by \(\mathrm{Reg}(\mathcal{I}) = J(\x\^0) - J(\x^\star)\). As RHAPD-S uses OGD for initialization, using the bound from \cite[Theorem 4]{li2020online} for the initialization regret of OGD, we arrive at the following theorem:

   \begin{theorem}[\textbf{Regret of RHAPD-S}]\label{regret_RHAPD-S}
	   Under assumptions \ref{bounded_subgradient}, \ref{smooth}, with OGD step size \(\eta = \frac{1}{l}\), the regret of RHAPD-S is bounded by \begin{align}\nn
		   \mathrm{Reg}(\textsc{RHAPD-S}) \leq \frac{\delta}{\left(1 + \frac{2\mu \rho_s}{\beta_s^2}\right) ^ {W}} \sum_{t = 1}^{N} \norm{\theta_t - \theta_{t - 1}},
	   \end{align} where \(\delta := \left(\frac{{\beta_s}}{l} + 1\right)\frac{G}{1 - \kappa}\), \(\kappa := \sqrt{1 - \frac{\mu}{l}}\), \(\theta_t = \argmin_{x \in \cX}f_t(x)\), \(\theta_{0} = x_{0}\); \(\beta_s^2 := 2 \left(l + \gamma + \frac{1}{\tau}\right)^2\),  \begin{align}\nn \rho_s := \min \bigg\{ \min_{t \in [1, N - 1]} \bigg(\frac{1}{\tau} - \frac{l_t}{2} + \gamma \bigg), \frac{1}{\tau} - \frac{l_N}{2} + \frac{\gamma}{2} \bigg\},\end{align} and the step size \(\tau\) is such that \(\rho_{s} > 0\).
   \end{theorem}
   \begin{proof}
	Applying the PL inequality, we establish the following: \begin{align}\nn
		J(\x ^ {(W)}) - J(\x ^ {*}) \leq \frac{J(\x ^ {(0)}) - J(\x ^ {*})}{\bigg(1 + \frac{2\mu \rho_{s}}{\beta_{s} ^ {2}}\bigg) ^ W}.
	\end{align}
	Note that this is similar to \eqref{opt_gap_bound}, however with different constants \(\rho_{s}, \beta_{s}\). It follows from \cite[Theorem 4]{li2020online} that the regret of the OGD initialization policy is bounded as 
	\begin{align}\nn
		\text{Reg}(\mathcal{I}) \leq \delta \left(\sum_{t = 1} ^ {N} \norm{\theta_{t} - \theta_{t - 1}}\right).
	\end{align} 
	
Proceeding similarly to the proof of Theorem \ref{dynamic_regret_non_smooth_general_g} completes the proof.
   \end{proof}
	
\begin{rem}\label{bound_comparsion} It is instructive to compare the dynamic regret bound for RHAPD-S with those obtained for RHGD and RHAG in \cite{li2020online}. From Theorem \ref{regret_RHAPD-S}, we have \begin{align}\nn
			\mathrm{Reg}(\text{RHAPD-S}) \leq \frac{\text{Reg}(\mathcal{I})}{\left(1 + \frac{2\mu\rho_{s}}{\beta_{s} ^ {2}}\right) ^ {W}}.
		\end{align}
		Also from \cite[Theorem 3]{li2020online}, we have the following bounds on the regret attained by RHGD and RHAG \begin{align}
			\mathrm{Reg}(\text{RHGD}) \leq Q_{f}\bigg(1 - \frac{1}{Q_{f}}\bigg) ^ {W} \text{Reg}(\mathcal{I}), 
			\quad \mathrm{Reg}(\text{RHAG}) \leq 2\bigg(1 - \frac{1}{\sqrt{Q_{f}}}\bigg) ^ {W} \text{Reg}(\mathcal{I}), \nn
		\end{align}
		where \(Q_{f} := \frac{l + 4\gamma}{\mu}\) denotes the condition number of the overall objective \(J(\x)\). For simplicity, let us set $\tau = \frac{1}{l}$ for the proposed algorithm RHAPD-S. A careful analysis of these bounds reveals that the bound of RHAPD-S is better than the bound of RHGD when \(\gamma \geq \frac{1}{4}\left(\mu+3l+\sqrt{\mu^2+14\mu l + 65l^2}\right)\),  which lies in the range $[2.76l,~3.24l]$ when \(\mu = l\). In Figures \ref{fig:gamma_25}, \ref{fig:gamma_300}, and \ref{fig:constrained_case}, we observe that the performance of RHAPD-S is better than RHGD, consistent with the bounds. However, as we decrease \(\gamma\), the bound of RHGD can turn out to be better for certain values of \(W\). This can be verified by setting \(\gamma = 0.1l, \mu = l, W \geq 1\). However, as observed in the numerical experiments detailed in Section \ref{sec:exp}, even in the small \(\gamma\) regime (where the bound of RHGD can turn out to be better than RHAPD-S), the empirical performance of RHAPD-S is better than RHGD (see Figure \ref{fig:gamma_0.1}). Compared with RHAG however, the bound of RHAG is better than that of RHAPD-S for all \(\gamma > 0\). However, the empirical performance of RHAPD-S still turns out to be better than RHAG under certain choices of \(\gamma\) (see Figures \ref{fig:gamma_0.1}, \ref{fig:ott}). Whether the bounds developed here can be tightened so as to completely explain the empirical behavior remains an open problem and will be investigated as future work. 
\end{rem}

\begin{table*}[t]
\centering
\begin{tabular}{|c|p{30em}|p{8em}|}
\hline
Parameter & Description & First Definition \\ \hline
$G$         &   Upper bound on subgradient of stage cost $f_t$    &   Assumption \ref{bounded_subgradient}  \\ \hline
$\theta_t$          &  Minimizer of $f_t$ over $\mathcal{X}$     & Equation \ref{path}    \\ \hline
$\mu$          &  Stage cost $f_t$ is $\mu_t$-strongly convex and $\mu \le \min_t \mu_t$       & Assumption \ref{non_smooth}   \\ \hline
$l_g$   &  Switching cost function $g$ is $l_g$-smooth     & Assumption  \ref{g_convex}   \\ \hline
$\gamma$       & Switching cost  \(g\) is s.t. $0 \le g(x,y) \leq \gamma/2\norm{x-y}^2$        & Assumption \ref{g_convex}   \\ \hline
$l$        &  Stage cost $f_t$ is $l_t$-smooth and $l = \max_t l_t$       & Assumption  \ref{smooth} \\ \hline        
$\tau$          &  Fixed step size      & Section \ref{apgmns}    \\ \hline
$\rho$          &  (RHAPD) Sufficient decrease constant of \(J(\x)\) &  Section \ref{sec:regret-ns}     \\ \hline
$\beta$          &  (RHAPD) Subgradient bound for $\Jt(\x)$ &  Section \ref{sec:regret-ns}      \\ \hline

$\rho_q$        & (RHAPD) Sufficient decrease constant of $J(\x)$ for quadratic $g$ &  Theorem \ref{rhapd_quadratic_regret}    \\ \hline
$\beta_q$        & (RHAPD) Subgradient bound of $\Jt(\x)$ for quadratic $g$ &  Theorem \ref{rhapd_quadratic_regret}      \\ \hline
$\rho_s$        &  (RHAPD-S) Sufficient decrease constant of $J(\x)$     &  Theorem \ref{regret_RHAPD-S}   \\ \hline
$\beta_s$        &  (RHAPD-S) Subgradient bound of $\Jt(\x)$ &   Theorem \ref{regret_RHAPD-S}       \\ \hline
$\eta$           &  OGD step size $\eta = 1/l$    &  Section \ref{sec:initialization}    \\ \hline

\end{tabular}
\label{params}
\caption{{Description of all the parameters across all the assumptions, theorems, and experiments.}}
\end{table*}

\begin{rem}\label{rem:uniformly_convex}
More generally, we can consider the case when \(f_{t}\) is uniformly convex, i.e., {$f_t$} satisfies \begin{align*}
		f_t(y) \geq f_t(x) + \ip{\partial f_t(x)}{y - x} + K \norm{y - x} ^ {p},
	\end{align*}
	for all $x, y \in \Rn^d$ and where \(p \ge 1, K > 0\). While the case of strongly convex stage costs considered here corresponds to $p = 2$, the proposed analysis in this paper can be readily extended to the case when \(f_{t}\) is $p$-uniformly convex for $p > 2$. 
\end{rem}

\section{Experiments} \label{sec:exp}
This section compares the performance of the proposed algorithms against that of the existing algorithms on tasks related to regression, trajectory tracking, and economic power dispatch. For each experiment, {we plot} the variation of the dynamic regret defined in \eqref{regret} against the lookahead window size $W$, while keeping the other parameters constant.  For the proximable stage costs case, we compare the performance of RHAPD (Algorithm \ref{RHAPD}) with that of PGD\footnote[2]{As mentioned earlier, the PGD and FISTA algorithms can be implemented in an online fashion} \eqref{prox_update} and its accelerated version FISTA. When the switching cost is quadratic, {we also} include the performance of the RHAM algorithm, which has been shown to be the special case of RHAPD in Section \ref{RHAM}. When the stage costs are smooth, we compare the performance of RHAPD-S (Algorithm \ref{RHAPD_Smooth}) with that of RHGD and RHAG from \cite{li2020online}. We also include the plots corresponding to RHAPD and FISTA whenever our smooth stage costs happen to be proximable. The step sizes used {for} various algorithms {in} the considered experiment {settings} are chosen as per Table \ref{tt_steps}.
{
\begin{itemize}
    \item For RHAPD-S, as mentioned in Theorem \ref{regret_RHAPD-S} earlier, the step size \(\tau\) is chosen such that \(\rho_{s} = \frac{1}{\tau} - \frac{l}{2} + \frac{\gamma}{2} > 0\). The choice \(\tau = \frac{1}{l}\) ensures that \(\rho_{s} > 0\) and was empirically the best-performing value across all the considered experiments.
     \item For RHAPD, as mentioned  in Theorem \ref{dynamic_regret_non_smooth_general_g}, the step size \(\tau\) is chosen such that \(\rho = \frac{\mu}{2} + \frac{1}{\tau} - l_{g} > 0\). In the special case when the switching cost is quadratic, as mentioned in Theorem \ref{rhapd_quadratic_regret}, \(\tau\) is chosen in a manner that \(\rho_{q} = \frac{\mu}{2} + \frac{1}{\tau} - \gamma > 0\). For the experiments with quadratic switching cost, the choice of \(\tau = \frac{0.8}{\gamma}\) ensures that \(\rho_{q} > 0\). Further, we found that this was the best-performing value for RHAPD across all the considered experiments with the quadratic switching cost. For the task of lasso regression with sum-squared switching cost (\textbf{E2}), we have \(g(x_{t}, x_{t - 1}) = \frac{\gamma}{2\sqrt{2}d} \ip{x_{t} - x_{t - 1}}{\mathbf{1}_{d}} ^ {2}\). It can be shown that \(g(x_{t}, x_{t - 1})\) is \(\gamma\)-smooth on \(\Rn ^ d \times \Rn ^ d\) (refer Appendix \hyperref[app:smooth]{H}), and therefore choosing \(\tau = \frac{0.8}{\gamma}\) ensures that \(\rho > 0\).
    \item For PGD and FISTA, the step sizes are chosen as per \cite{beck2009fast}.
    \item For RHGD and RHAG, the step sizes are chosen as per \cite{li2020online}.
\end{itemize}
}

\begin{table}[htb] 
	\centering
	\normalsize
	\begin{tabular}{|c|c|} \hline
		Algorithm 	& Parameters \\\hline
		RHAPD-S		& $\tau = 1/l$   \\\hline
		RHAPD 		& $\tau = 0.8/\gamma$ \\\hline
		PGD 		& $\tau = 1/L_{\nabla H}$ \\\hline
		FISTA   	& $\tau = 1/L_{\nabla H}$\\\hline
		RHGD 	&  $\eta_G = 1/L_{\nabla J}$ \\\hline
		RHAG &	$\lambda = \frac{\sqrt{L_{\nabla J}} - \sqrt{\mu}}{\sqrt{L_{\nabla J}} + \sqrt{\mu}}$ \\\hline 
	\end{tabular}
	\caption{Various step sizes used in the experiments. See \eqref{FHdef} for the definition of $H$. {Note that \(L_{\nabla H}, L_{\nabla J}\) denote the Lipschitz-smoothness constants for \(H\) and \(J\) respectively.} }\label{tt_steps}
\end{table} 

For all experiments, the {plots also} show the performance of MPC, where at time $t$, we solve the $W$-stage optimization problem
\begin{align}\label{MPC} 
\{x_j\^t\}_{j=t}^{t + W - 1} \!=\hspace{-5mm} \argmin_{x_t \ldots, x_{t + W - 1} \in \cX} \hspace{-1mm}\sum_{\tau = t} ^ {t + W - 1}\hspace{-2mm} f_{\tau}(x_{\tau}) + g(x_{\tau}, x_{\tau - 1}),
\end{align}
and take the action corresponding to $x_t\^t$. While the computational cost of solving \eqref{MPC} is much higher than that of the first-order methods, the performance of MPC serves as a benchmark in various settings. {

\begin{figure*}[!tbp]
	\begin{minipage}{0.32\textwidth}
		\centering
	  \includegraphics[width=\linewidth]{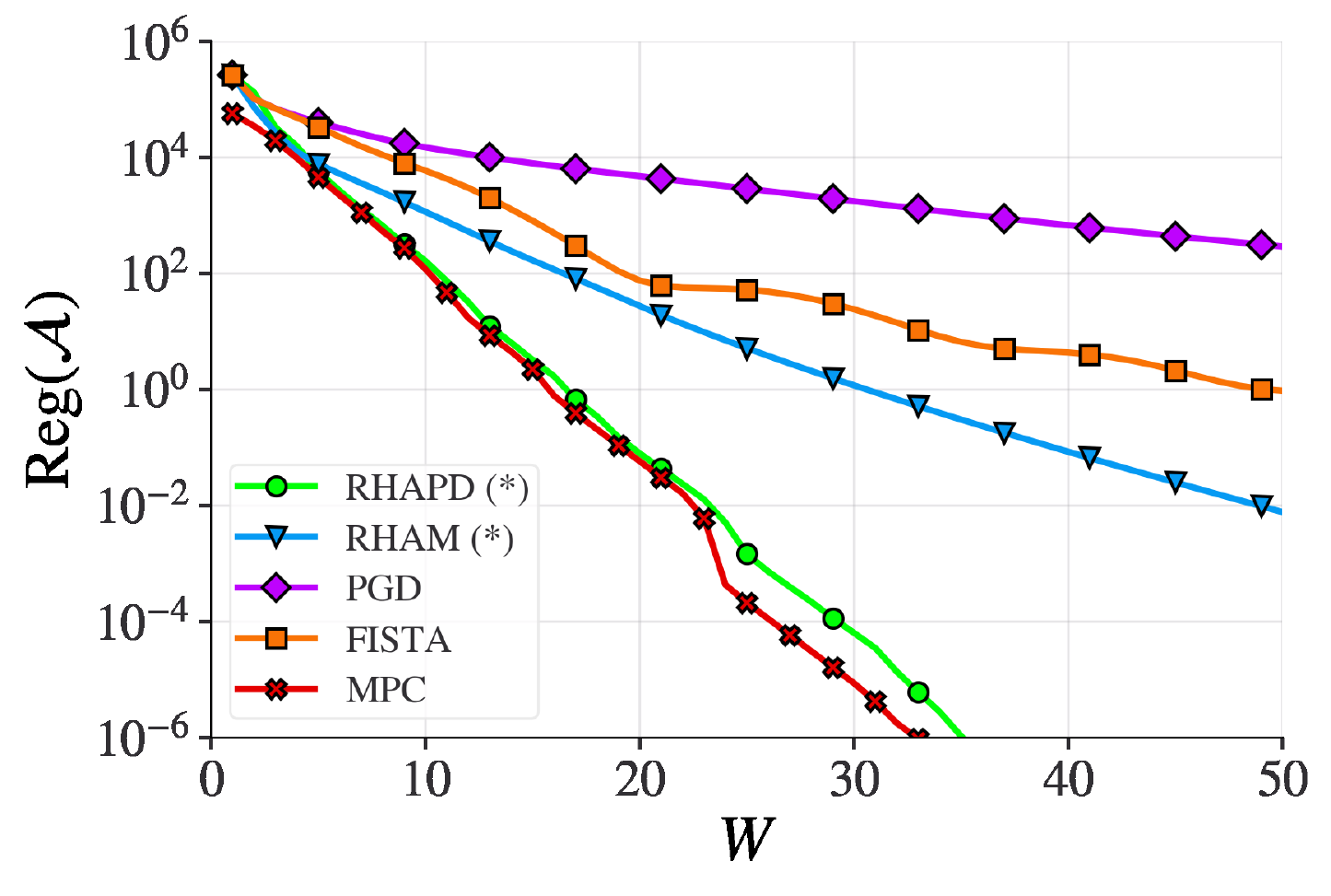}
	  \caption{Results for \textbf{E1}}
	  \label{fig:non_smooth_lasso_quad}
	\end{minipage}
	\begin{minipage}{0.32\textwidth}
		\centering
		\includegraphics[width=\linewidth]{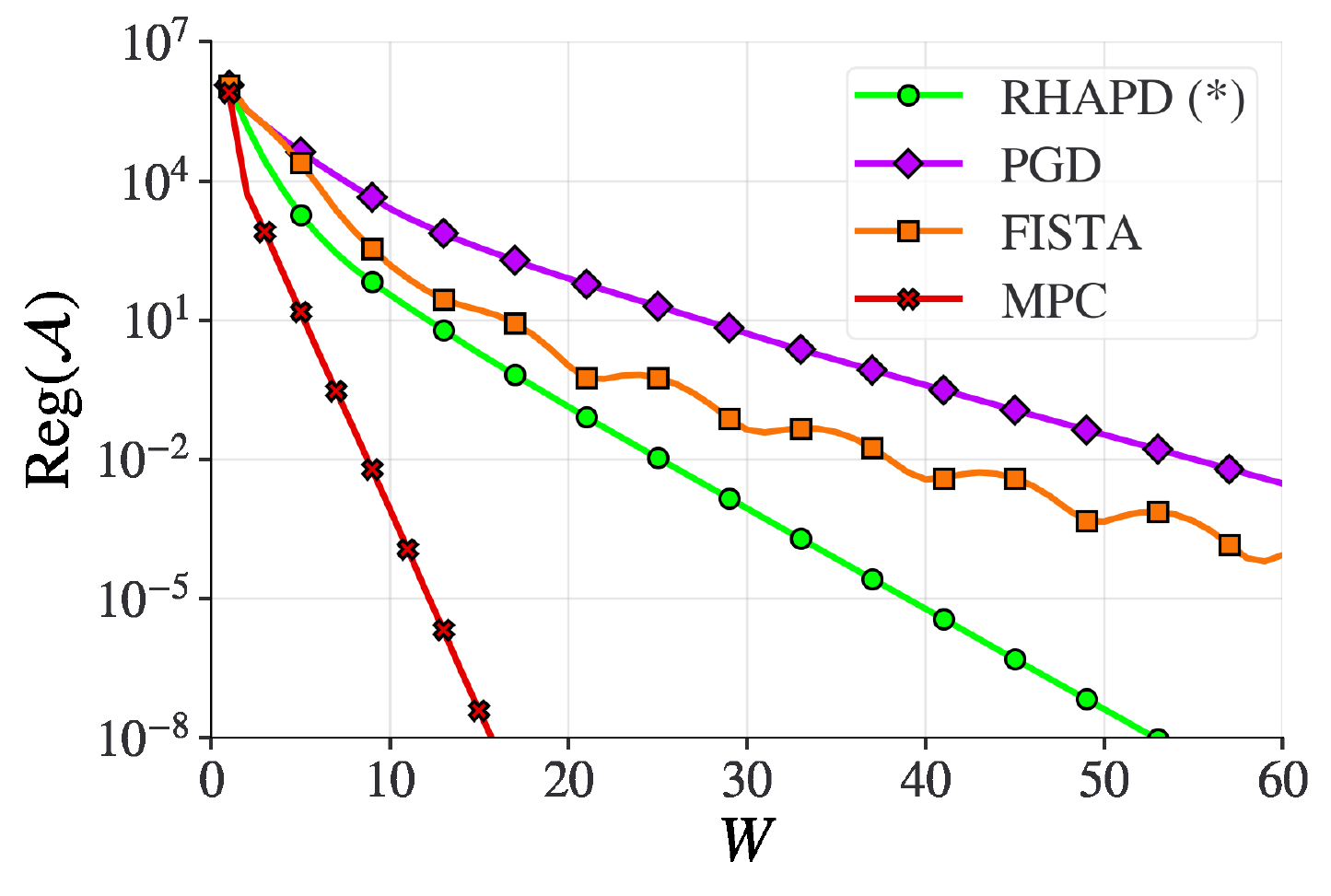}
		\caption{Results for \textbf{E2}}
		\label{fig:non_smooth_lasso_non_quad}
	\end{minipage}
	\begin{minipage}{0.32\textwidth}
		\centering
		\includegraphics[width=\linewidth]{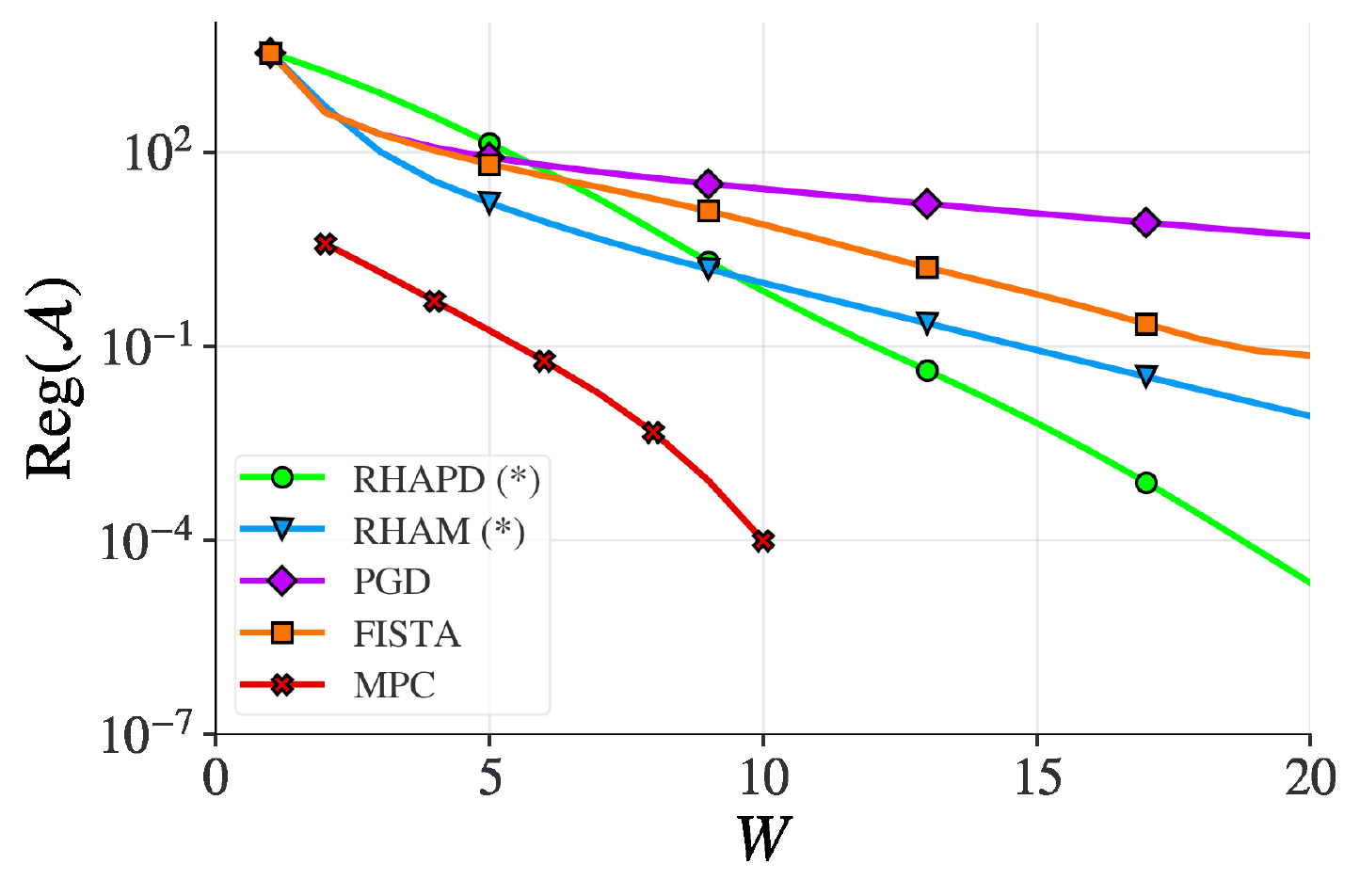}
		\caption{Results for \textbf{E3}}
		\label{fig:overlapping_group_lasso}
	\end{minipage}
\end{figure*}

\begin{table*}[tb]
	\centering
	\begin{tabular}{|c|c|c|c|c|c|c|c|c|c|c|}
	\hline
	Exp 	& $d$ 	& $N$ 	& $\gamma$ 	& $\cX$ 	& $x_0$ 	& $\sigma$ 	& Objective 	& $L_{\nabla H}$ & Other Parameters \\ \hline
	\textbf{E1}       	& $1$  	& $100$  	& $10$     	& $[-10^{5}, 10^5]$	& 0	& 1000			& \eqref{eq:problem_E1} & $4\gamma$ & $M = 60$, $\lambda=50$               \\ \hline
	\textbf{E2}      	& $2$  	& $100$  	& $2$        	&  $[-100, 100] \times [-10, 10]$ 	&  $\mathbf{0}$  		& 100			& \eqref{eq:problem_E2} & $2\sqrt{2}\gamma$ &$M=1$, $\lambda = 1$         \\ \hline
	\textbf{E3}        	& $50$  	&  $30$ 	& $10$         	&  ${[-10^{7}, 10^{7}]}$			&  $\mathbf{0}$  		& $1$			&  \eqref{eq:problem_E3} & $4\gamma$ &$G_1 \ldots G_9$\\ \hline
	\end{tabular}
	\caption{All parameters used in Experiments for Section \ref{reg_exp}. For \textbf{E1} and \textbf{E3}, \(L_{\nabla H} = 4\gamma\) follows from \cite[Lemma 1]{li2020online}. For the computation of \(L_{\nabla H}\) for \textbf{E2}, refer Appendix \hyperref[app:overall_smooth]{I}.}\label{table_reg}
	\end{table*}
For some of the numerical experiments, we also report the overall runtimes of all the algorithms averaged across 5 runs for a fixed prediction window size $W=10$. Note that the overall runtime of MPC depends on the optimization algorithm used to solve \eqref{MPC} and is problem specific. Therefore, for simplicity, we only consider reporting the runtimes for the problem of lasso regression with quadratic switching cost (\textbf{E1}), in which case we use the alternating minimization algorithm (Section \ref{RHAM}) to solve \eqref{MPC}, demonstrating that the proposed algorithms are much more computationally inexpensive. Further, we report the runtimes for the problem of trajectory tracking (\textbf{E2}) as well, demonstrating that the proposed algorithms offer better performance than RHGD and RHAG at similar runtime costs.
All the experiments were conducted on {an Ubuntu 20.04 machine with an Intel Core i5-9300H CPU.} 

\subsection{Lasso Regression} \label{reg_exp}
We consider a generic setting where an agent seeks to solve a series of regression tasks $\{T_t\}_{t=1}^N$ in an online fashion, while also minimizing deviations in its state $x_t$. We begin by considering the following two tasks. 

\noindent\textbf{E1}. \emph{Lasso Regression with quadratic switching cost. }
\begin{align}\label{eq:problem_E1} \tag{$\mathcal{P}_{E1}$}
	f_t(x_t) = \frac{1}{M} \sum_{j = 1}^M \norm{x_t - u\^j}^2 + \frac{\lambda}{2} \norm{x_t}_1, \quad 
g(x_t, x_{t - 1}) = \frac{\gamma}{2} \norm{x_t - x_{t - 1}}^2,
\end{align}
where $u\^j \sim \cN(\mathbf{0}, \sigma^2 \I)$ for $1 \leq j \leq M$. We take $N = 100$, $M = 60$, $d = 1$, $\sigma = 1000$, $\lambda = 50$, $\gamma = 10$, and $x_0 = 0$. The feasible set $\cX$ is taken to be $[-10^5, 10^5]$.

\noindent\textbf{E2}. \emph{Lasso Regression with sum-squared switching cost. }
\begin{align}\label{eq:problem_E2}\tag{$\mathcal{P}_{E2}$} 
	f_t(x_t) = \frac{1}{M} \sum_{j = 1}^M \norm{x_t - u\^j}^{2} + \frac{\lambda}{2} \norm{x_t}_1, \quad
	g(x_t, x_{t - 1}) = \frac{\gamma}{2\sqrt{2}d} \ip{x_t - x_{t - 1}}{\mathbf{1}}^2,
\end{align}
where $u\^j \sim \cN(\mathbf{0}, \sigma^2 \I)$ for $1 \leq j \leq M$, and we take $M=1$, $d=2$, $\sigma=100$,  $\lambda=1$, $\gamma=2\sqrt{2}$, $N = 100$, and $x_0 = \mathbf{0}$. The feasible set \(\cX = [-100, 100] \times [-10, 10]\) is a rectangle in \(\Rn^2\). The switching cost function in \eqref{eq:problem_E2} penalizes changes in $\ip{\mathbf{1}}{x_t}$ over consecutive time steps but allows changes in individual elements of $x_t$. From the Cauchy-Schwarz inequality, it can be seen that $g$ satisfies Assumption \ref{g_convex}. 

The results of experiments \textbf{E1} and \textbf{E2} are presented in Figure \ref{fig:non_smooth_lasso_quad} and \ref{fig:non_smooth_lasso_non_quad}, respectively. In \textbf{E2}, since the switching cost is not quadratic, we do not include the performance of RHAM. In both \textbf{E1} and \textbf{E2}, the proposed RHAPD performs better than PGD and its accelerated variant. Interestingly, the performance of RHAPD in Figure \ref{fig:non_smooth_lasso_quad} is close to that of MPC, suggesting its superior performance for this case. {Table \ref{runtimes:lasso} reports the runtime of all the algorithms for \textbf{E1}. As can be observed from Table \ref{runtimes:lasso}, RHAPD provides improved performance than PGD and FISTA at {a similar computational expense}. Further, the performance of RHAPD is comparable to MPC while the runtime of RHAPD is significantly lower than that of MPC. However, we remark that these runtimes are implementation dependent and may not generalize to other settings. }


\begin{table}[htb] 
	\centering
	\normalsize
	\begin{tabular}{|c|c|} \hline
		Algorithm 	& Runtime (ms) \\\hline
		\textbf{RHAPD}		& \textbf{44} \\\hline
		\textbf{RHAM} 		& \textbf{44} \\\hline
		PGD 		& 39 \\\hline
		FISTA   	& 44 \\\hline
		MPC 	&  13476 \\\hline
	\end{tabular}
	\caption{{Runtime (rounded to the nearest integer) for algorithms in \textbf{E1} for $W=10$, averaged over 5 runs.}}\label{runtimes:lasso}
\end{table} 

We next consider a more complicated setting where the proximal cannot be calculated readily. \\
\noindent\textbf{E3}. \emph{Overlapping Group Lasso Regression with quadratic switching cost }
   \begin{align}\label{eq:problem_E3}\tag{$\mathcal{P}_{E3}$}
	f_t(x) = \frac{\norm{x - u}^{2}}{2} + \sum_{i = 1} ^{g} \norm{x_{G_i}}_{2}, \quad
	g(x_t, x_{t - 1}) = \frac{\gamma}{2} \norm{x_t - x_{t - 1}}^{2}.
\end{align}
The formulation in \eqref{eq:problem_E3} is borrowed from that in \cite{yuan2011efficient}. We take $d=50$ and $G_i = \{5i-4, \ldots, 5i+5\}$ for $1\leq i \leq 9$, so that groups $G_j$ and $G_{j+1}$ have 5 overlapping indices for $1 \leq j \leq 8$. The feasible \(\mathcal{X}\) is taken to be a large interval \({\mathcal{X} = [-10 ^ {7}, 10 ^ {7}]}\). As earlier, we take $u\sim \cN(\mathbf{0}, \I)$, with $N = 30$ and $\gamma = 10$. Here, since the {$\prox$} cannot be obtained in closed form, we solve it using the Alternating Directions Method of Multipliers (ADMM) method, as implemented in \cite{yuan2011efficient}. The ADMM algorithm is terminated when the norm of the residual of both the primal and dual feasibility is below $10^{-10}$.

The results for experiment \textbf{E3} are presented in Figure \ref{fig:overlapping_group_lasso}. As observed in the earlier experiments, the proposed algorithm RHAPD and its variant RHAM outperform PGD and FISTA. While the need to run ADMM at each iteration makes all the algorithms computationally intensive in this case, we still found them to be significantly faster than MPC. 

Table \ref{table_reg} summarizes all the parameters used in the experiments for Section \ref{reg_exp}.
 
\begin{figure*}[tb]
	\begin{minipage}{0.32\textwidth}
		\centering
		\includegraphics[width=\linewidth]{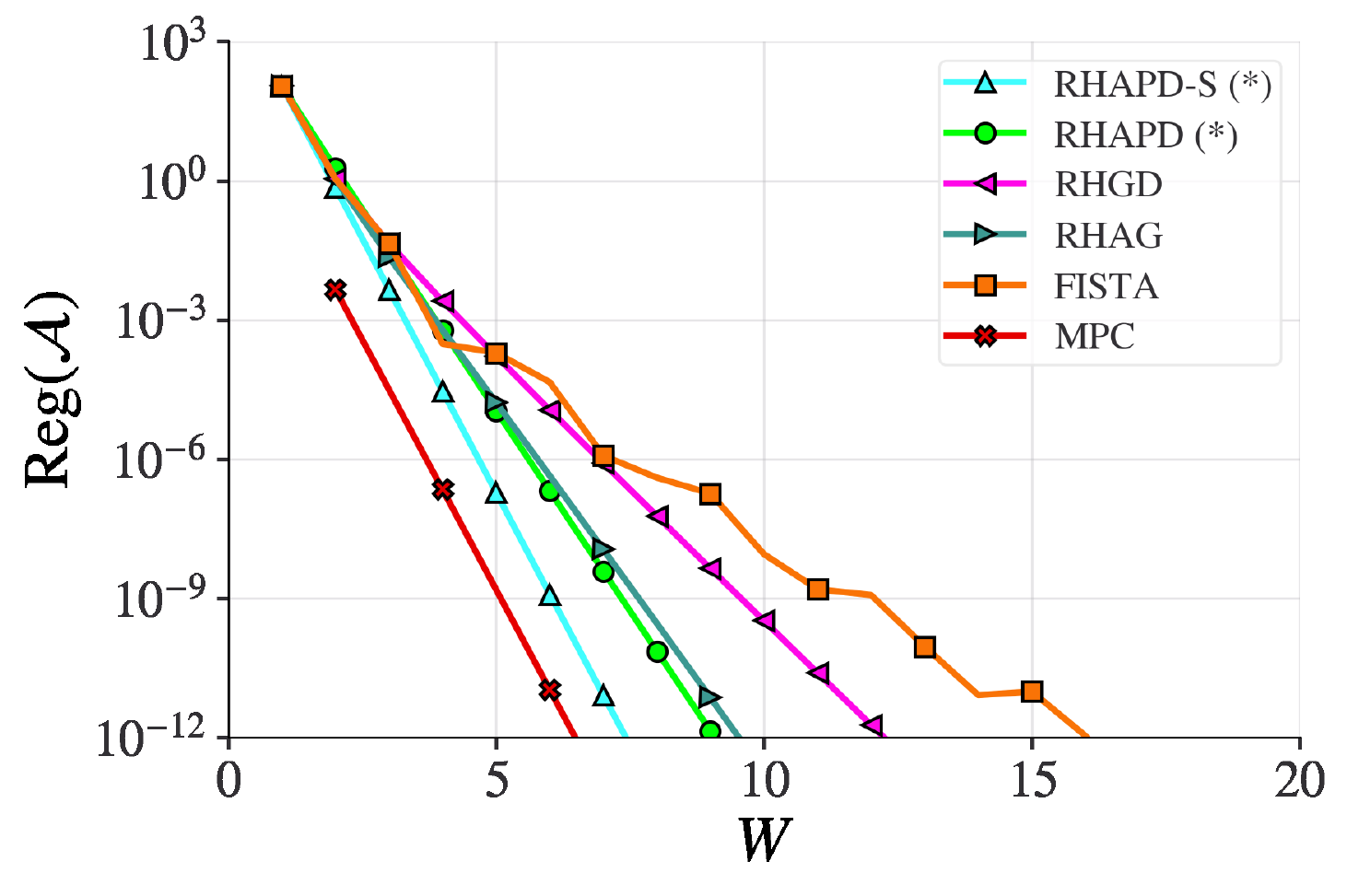}
		\caption{Results for \textbf{E4}, $\gamma = 0.1$}
		\label{fig:gamma_0.1}
	\end{minipage}
	\begin{minipage}{0.32\textwidth}
		\centering
		\includegraphics[width=\linewidth]{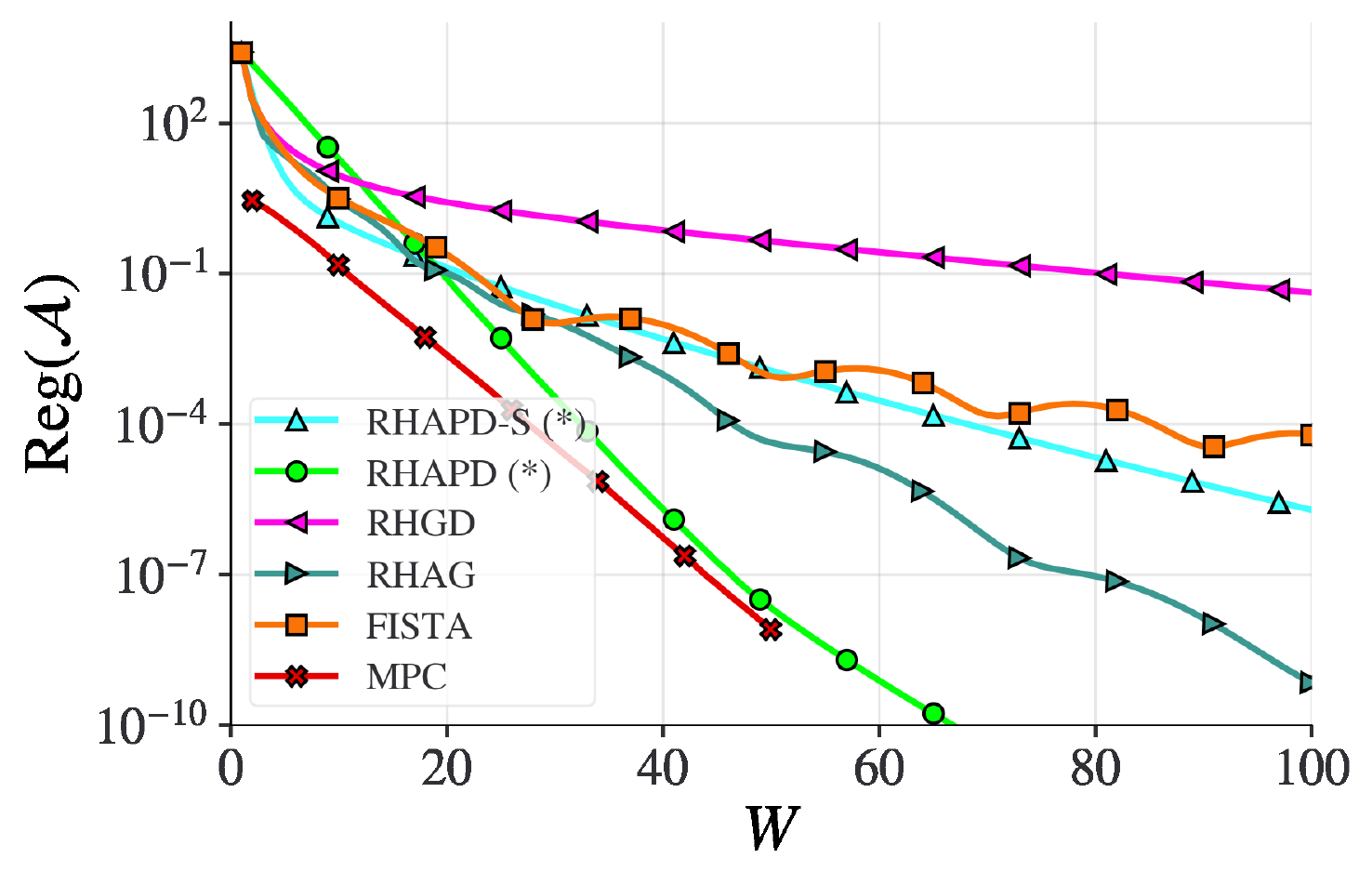}
		\caption{Results for \textbf{E4}, $\gamma = 25$}
		\label{fig:gamma_25}
	\end{minipage}
	\begin{minipage}{0.32\textwidth}
		\centering
		\includegraphics[width=\linewidth]{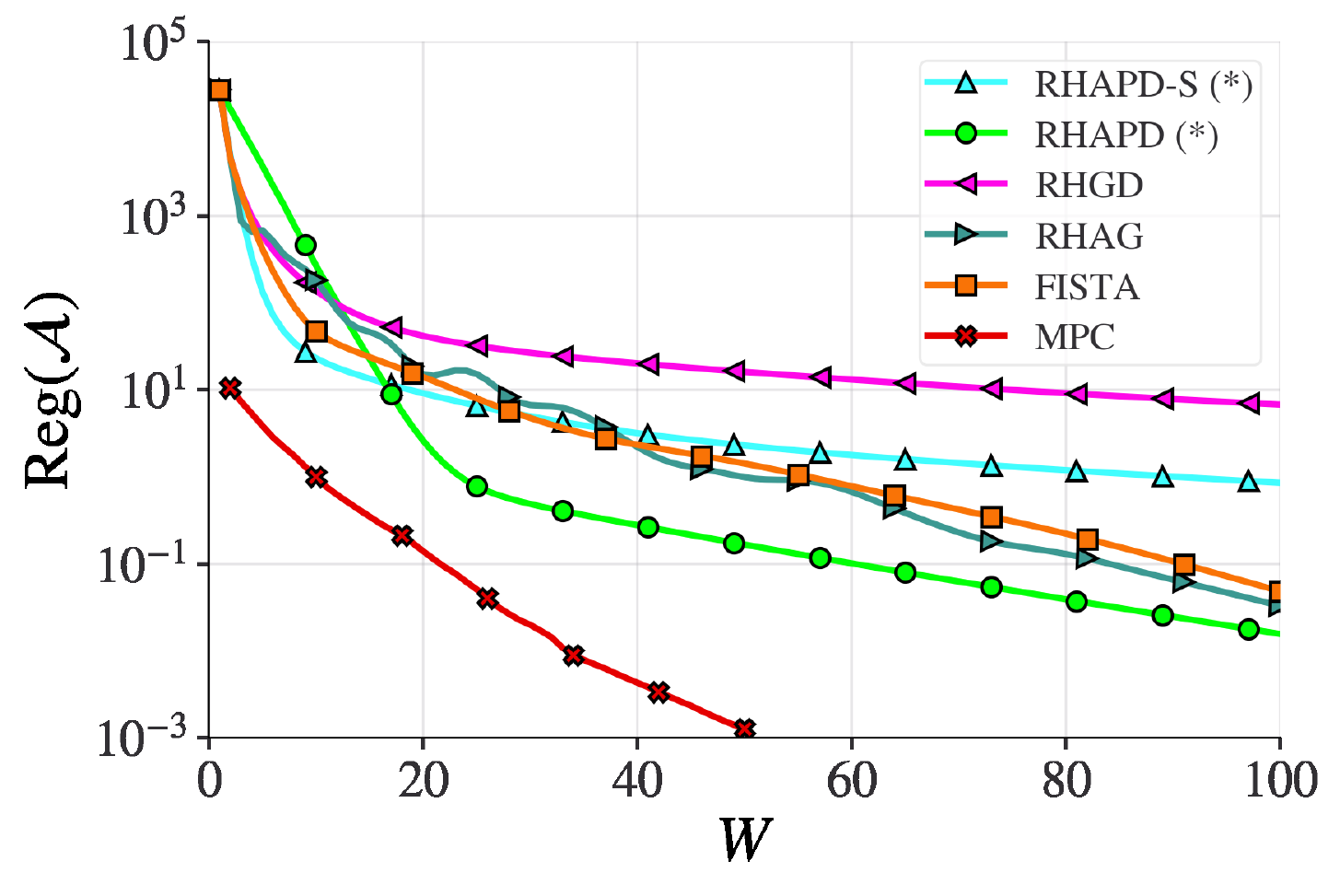}
		\caption{Results for \textbf{E4}, $\gamma = 300$}
		\label{fig:gamma_300}
	\end{minipage}
\end{figure*}

\subsection{Trajectory Tracking}\label{sec:tt}
We now consider the trajectory tracking problem, where a robot is required to track a target whose position at time $t$ is denoted by $u_t$. The variable $x_t$ captures the robot location and the stage cost is the squared distance from the target. A quadratic switching cost is chosen so as to minimize the energy consumption of the robot. With these choices, we formulate \eqref{eq:problem} with
\begin{align}
	f_t(x_t) = \frac{1}{2}\norm{x_t - u_t}^{2}, \quad
	g(x_t, x_{t - 1}) = \frac{\gamma}{2} \norm{x_t - x_{t - 1}}^{2}. \nonumber
\end{align}
The smooth stage costs are chosen to allow us to compare the proposed RHAPD-S (Algorithm \ref{RHAPD_Smooth}) with related algorithms such as RHGD and RHAG from \cite{li2020online}, as well as MPC. As the stage costs are also proximable, we also plot the performance of RHAPD (Algorithm \ref{RHAPD}) and FISTA.

For the sake of consistency, we initialize all algorithms using the OGD update rule with learning rate $\eta = 1/l$ as described in Section \ref{smoothOCO}. For experiment \textbf{E5} however, we set \(\eta = 0.4 \le \frac{1}{l}\), similar to \cite{li2020online}. The step sizes are chosen as in Table  \ref{tt_steps}, with $L_{\nabla H} = 4\gamma, L_{\nabla J} = l + 4\gamma$. We begin by considering the following specific problem.

\noindent \textbf{E4}. \emph{One-dimensional with large interval}. We take $d=1$ and the feasible region as the large interval $\cX = [-10^6, 10^6]$. Further, we take $x_0 \in \cN(0,1)$, $N = 100$, and $u_t \sim \cN(0,1)$ for $1 \leq t \leq N$. We compare the performance of various algorithms for $\gamma = 0.1$ (Figure \ref{fig:gamma_0.1}), $\gamma = 25$ (Figure \ref{fig:gamma_25}), and $\gamma = 300$ (Figure \ref{fig:gamma_300}).

We observe that for all three values of the considered \(\gamma\), either RHAPD or RHAPD-S outperforms all the other algorithms. For small $\gamma$, which allows for better tracking of the target, the performance of RHAPD-S is the best, in contrast to the order of dynamic regret bounds as observed in Remark \ref{bound_comparsion}. On the other hand, for a larger value of $\gamma$, which yields a more energy-efficient trajectory, the performance of RHAPD is the best, followed by that of RHAG and RHAPD-S. In the extreme case of $\gamma = 300$, where the robot stays close to the starting location for the duration of the experiment, the performance of RHAPD is still the best, followed by that of RHAG and FISTA. As $\gamma$ increases, the step size $\eta_G$ (Table \ref{tt_steps}) for RHGD decreases, and the momentum parameter $\lambda$ (Table \ref{tt_steps}) for RHAG increases, due to which we observe degradation in the performance of RHGD and improvement in the performance of RHAG. {This} experiment demonstrates that the problem parameters greatly influence the relative performance of various algorithms. 

{Table \ref{runtimes:track} highlights the low computational cost of the proposed algorithms for this particular setting, where the proximal has a closed form and is hence easy to compute. With almost similar runtime as RHAG, the proposed algorithm RHAPD is able to perform much better. Similarly, RHAPD-S performs better than RHGD despite having a similar runtime. The runtime of MPC was not included here as it was found to be several orders of magnitude higher than that of the other algorithms and varied greatly across different implementations. The higher runtime of MPC is also expected since it entails solving a \(W\)-stage optimization \(\mathcal{O}(N - W)\) times \cite{li2020online}.}

\begin{table}[!htb] 
	\centering
	\begin{tabular}{|c|c|} \hline
		Algorithm 	& Runtime (ms) \\\hline
		\textbf{RHAPD-S}		& \textbf{0.20} \\\hline
		\textbf{RHAPD} 		& \textbf{1.58} \\\hline
		RHGD 		& 2.95 \\\hline
		RHAG   	& 2.95 \\\hline
		FISTA 	&  3.74 \\\hline
	\end{tabular}
	\caption{Runtime for algorithms in \textbf{E4} ($\gamma = 25$) for $W=10$, averaged over 5 runs.}
    \label{runtimes:track}
\end{table}

We next consider a special example from \cite{li2020online} which is designed so as to extract the best possible performance of RHAG and is motivated from the lower bound analysis there. \cite{li2020online} propose this example to demonstrate that RHAG performs better than MPC (for \(W \leq \frac{N}{2}\)) at least in some special cases. 

\noindent \textbf{E5}. \emph{One-dimensional with small interval}. We take $d=1$ with $\cX = [0,6]$. Further, we set $N = 20$, $\gamma = 20$, and $\u = [6, 0, 6, 0, 6, 6, 0, 6, 6, 0, 6, 6,0, 6, 6, 6, 6, 6, 6, 6]$.

Results for \textbf{E5} are presented in Figure \ref{fig:constrained_case}, where it can be seen that RHAG performs slightly better than MPC for $W \leq 10$, as also shown in the experiments in \cite{li2020online}. We observe that for $W \leq 10$, RHAPD-S performs the best of all the algorithms while the performance of RHAPD is the worst. For this specific case, we also observe that for large values of $W$, the performance of RHAPD improves and is better than that of RHAPD-S, while the performance of RHAG degrades. However, this specific observation seems to be a property of the sequence constructed for this experiment and does not generalize to the other experiments.

\begin{figure*}[htb]
	\begin{minipage}{0.32\textwidth}
		\centering
		\includegraphics[width=\linewidth]{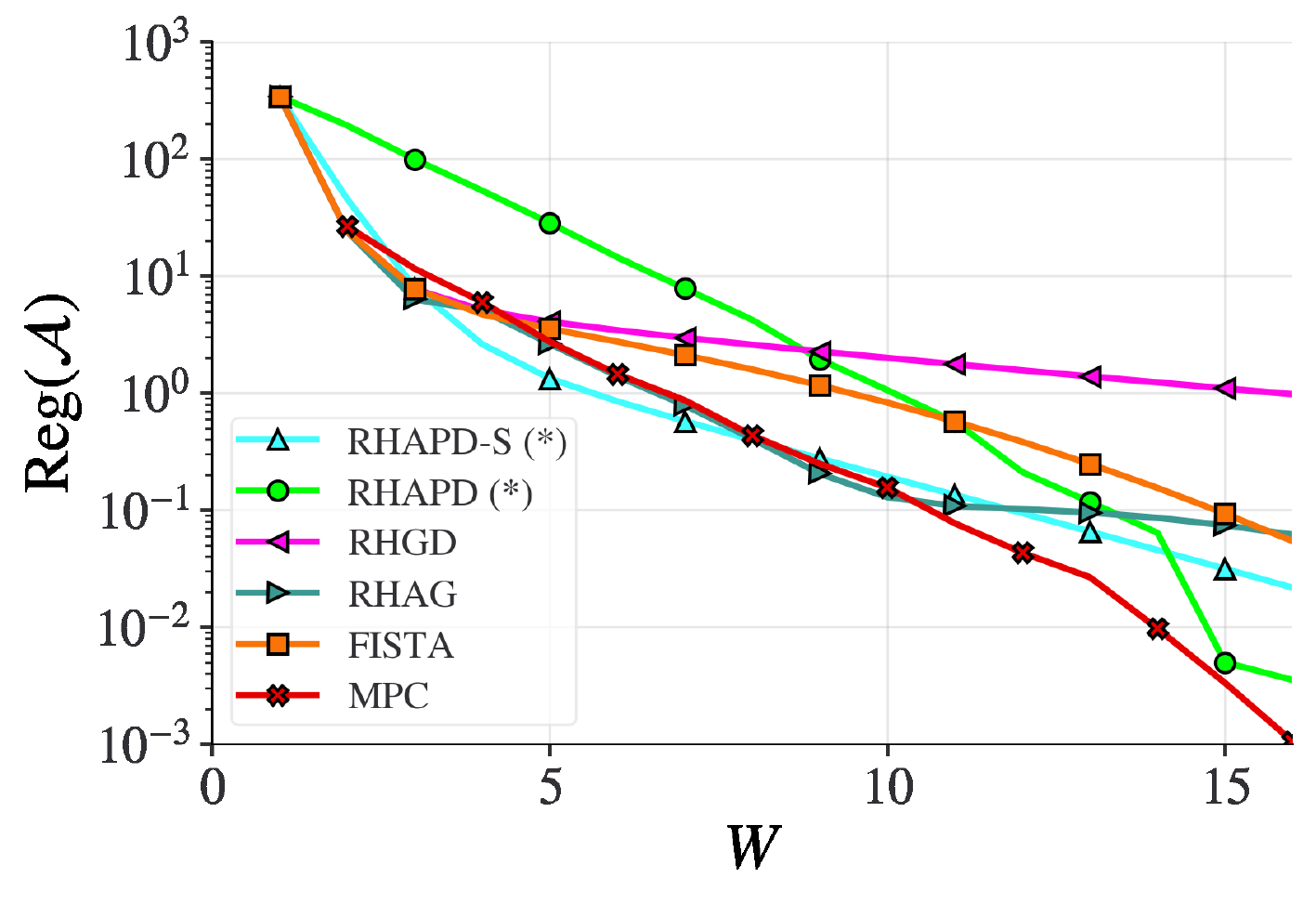}
		\caption{Results for \textbf{E5}}
		\label{fig:constrained_case}
	\end{minipage}
	\begin{minipage}{0.32\textwidth}
		\centering
		\includegraphics[width=\linewidth]{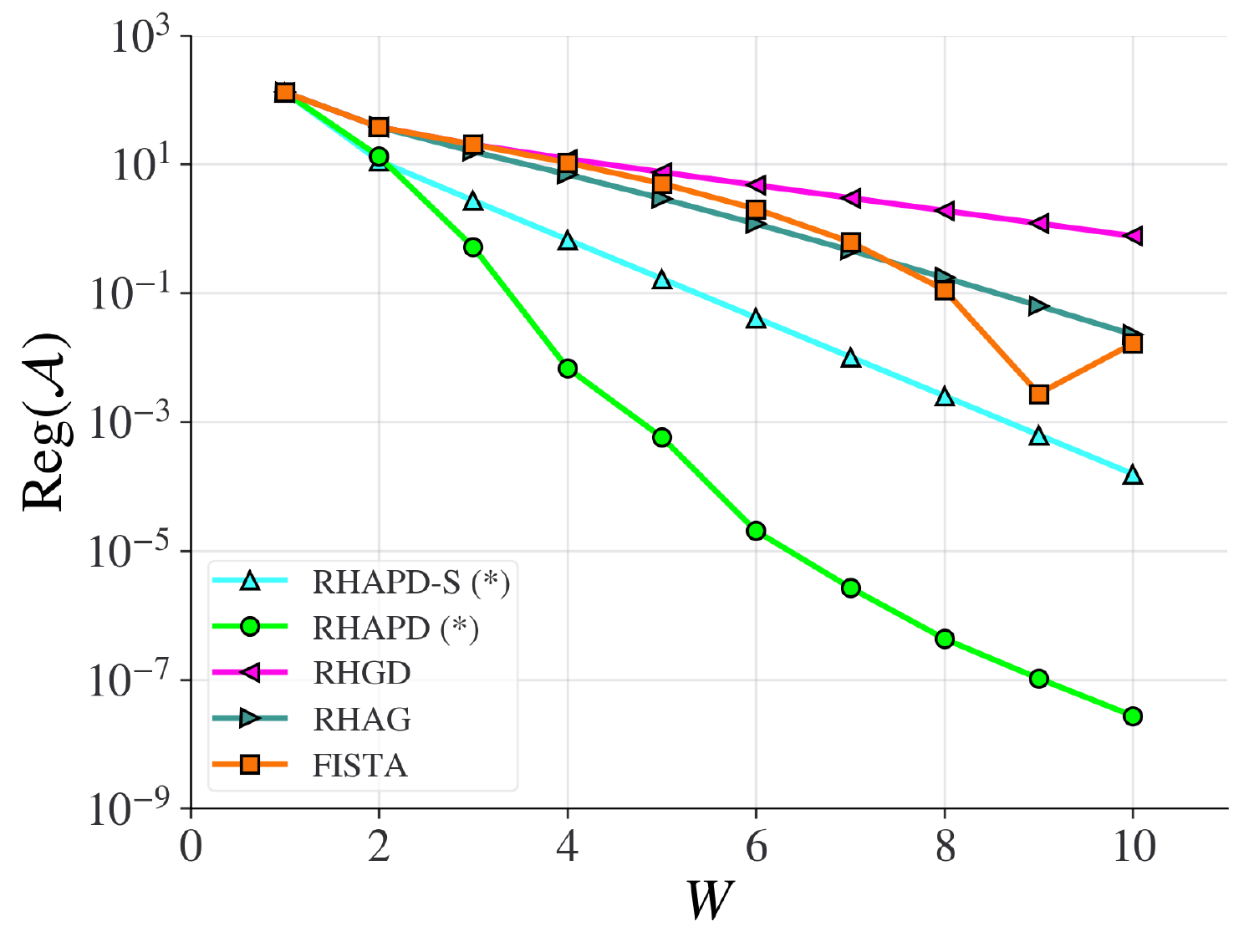}
		\caption{Results for \textbf{E6}}
		\label{fig:ott}
	\end{minipage}
	\begin{minipage}{0.32\textwidth}
		\centering
		\includegraphics[width=0.8\linewidth]{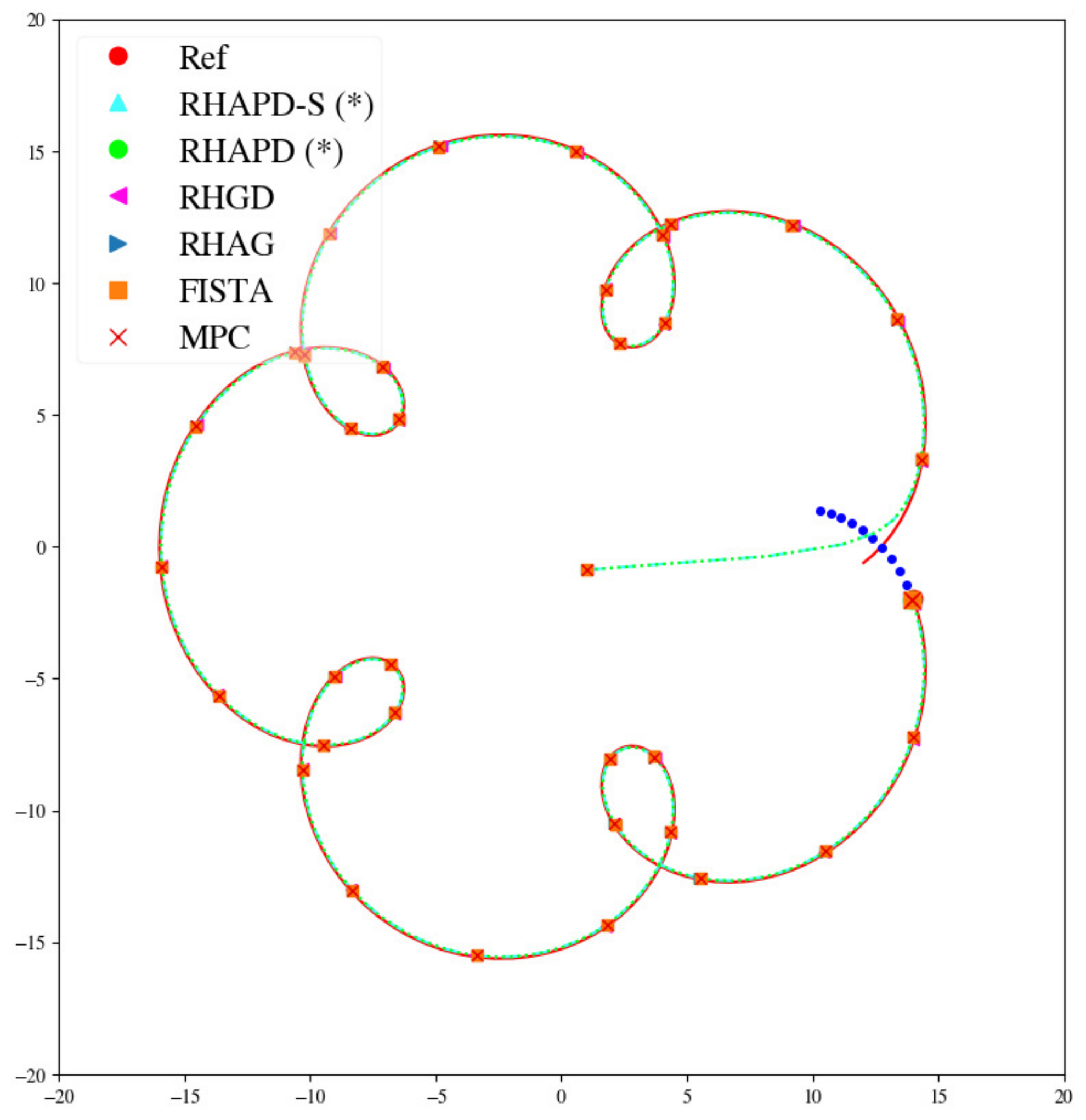}
		\caption{The trajectory tracked for \textbf{E6}}
		\label{fig:ott_curve}
	\end{minipage}
\end{figure*}

\begin{table*}[b]
	\centering
	\hspace*{-1.5cm}  
	\begin{tabular}{|c|c|c|c|c|c|c|c|c|c|c|c|}
		\hline
		Exp & $d$ & $N$ & $\gamma$ & $\cX$ & $u_t$ & $x_0$ & $\eta_{OGD}$ & $L_{\nabla H}$ &  $l = \mu$ & $L_{\nabla J}$ \\ \hline
		\textbf{E4}  & $1$  & $100$  & $\in \{0.1, 25, 300\}$ & $[-10 ^ {6}, 10 ^ {6}]$ & $\sim \mathcal{N}(0,1)$ & $\sim \mathcal{N}(0,1)$ & $1/l$ & $4\gamma$ &  $1$   & $l+4\gamma$  \\ \hline
		\textbf{E5}  & $1$  & $20$  & $20$ & $[0,6]$  &  as in \textbf{E5} & $0$ & $0.4$ & $4\gamma$      &  $1$   & $l+4\gamma$ \\ \hline
		\textbf{E6}  & $2$  & $300$  &  $1$        & $[-10^{6}, 10^{6}] \times [-10^{6}, 10^{6}]$      & as in \textbf{E6}         & $x_0^{(r)}, y_0^{(r)} \sim \mathcal{N}(0,1)$ & $1/l$ & $4\gamma$      &  $1$    & $l+4\gamma$  \\ \hline
	\end{tabular}
	\caption{Parameters for experiments in Section \ref{sec:tt}} \label{table_tt}
\end{table*}

Finally, we implement the proposed algorithms for tracking a well-defined reference trajectory in two dimensions.

\noindent \textbf{E6} \label{E6} \emph{Tracking target at $u_t = (x_{t} ^ {(r)}, y_{t} ^ {(r)}) = \big(12 \cos (t - 6) - 4 \cos(6(t - 6)), 12 \sin(t - 6) - 4 \cos(6(t - 6))\big)$ for $N = 300$, $W = 10$, and $\gamma = 1$}. 

The starting position is chosen as $u_0 \sim \cN(\mathbf{0}, \I)$, and {we assume} the robot is allowed to move freely within a large interval \(\mathcal{X} = [-10^{6}, 10^{6}] \times [-10^{6}, 10^{6}]\). The reference trajectory as well as the curves traced by the trackers are depicted in Figure \ref{fig:ott_curve}.  The red dot denotes the target $u_t$, while the tracker positions have their respective labels. The set of blue dots denotes the lookahead window $\{u_i\}_{i=t}^{t+W-1}$ available to the trackers. The dynamic regret achieved by the trackers is presented in Figure \ref{fig:ott}. In this case, we observe that both, RHAPD and RHAPD-S outperform RHAG and RHGD. Since the experiment is performed for a fixed $W$, the regret achieved by the MPC tracker, which is the least of all trackers, is not depicted in the figure. {Additionally, given} 
\(N = 300\), we found it quite computationally intensive to plot the variation of the regret vs \(W\) for MPC.

The animation for the experiment can be found \href{https://drive.google.com/file/d/1LKRpVNeCV7TS2cst6bZU_4Z0Hg3okYIw/view?usp=share_link}{here}. We note that all the compared trackers are able to track the trajectory mostly accurately, but there are clear differences in the performance when viewed closely enough, as is visible from the zoomed-in view in the bottom right corner of the animation. Table \ref{table_tt} summarizes the parameters used for the experiments in Section \ref{sec:tt}.

\subsection{Economic Power Dispatch} \label{eco_ds}
{
In this section, we demonstrate the efficacy of the proposed algorithms on real-life data. We consider the problem of economic power dispatch as formulated in \cite{li2018using}. In this problem, given a set of $k$ generators with outputs $(x_{t, 1}, \dots,  x_{t, k})$ at time $t$, the objective is to minimize the total generation cost $\sum_{i=1}^k c_i(x_{t,i})$ while also fulfilling the gap between the time-varying power demand $d_t$ and the renewable energy supply $s_t$. Although there may be random fluctuations in both demand and supply, forecasts and predictions based on past years' data are generally available for short time windows \cite{prediction_website_1, prediction_website_2}. Thus in our formulation of \eqref{eq:problem}, we have the stage cost function as the generation cost with an imbalance penalty, given by
}
\begin{align*}
    {f_t(x_t) = \sum_{i=1}^k c_i(x_{t,i}) + \xi_t \bigg(\sum_{i=1}^k x_{t,i} + s_t - d_t\bigg)^2,}
\end{align*}{
where $\xi_t > 0$ is a parameter that determines the cost incurred for not tracking the demand-supply gap adequately. In the context of power systems, significant operation costs are encountered when switching the power output of generators -- these are termed as ramp costs, and following the existing literature \cite{mookherjee2008dynamic, tanaka2006real, li2018using}, we model these switching costs as a quadratic function $g(x_t, x_{t-1}) = \frac{\gamma}{2}\norm{x_t - x_{t-1}}^2$. With this, we seek to minimize the total generation cost, including the imbalance penalty and the ramp costs over $N$ timesteps.
}

{
\noindent\textbf{E7.} In the interest of replicability, we consider the same data as in \cite{li2018using}. We are given $k=3$ generators, with quadratic generator cost functions}
\begin{align*}
    {c_{1}(x_{t, 1})} &= {x_{t, 1} ^ 2 + 15 x_{t, 1} + 10,} \\
    {c_{2}(x_{t, 2})} &= {1.2 x_{t, 2} ^ 2 + 10 x_{t, 2} + 27,}  \\
    {c_{3}(x_{t, 3})} &= {1.4 x_{t, 3} ^ 2 + 6 x_{t, 3} + 21.}
\end{align*}
\begin{figure*}[htb]
	\begin{minipage}{0.32\textwidth}
		\centering
		\includegraphics[width=\linewidth]{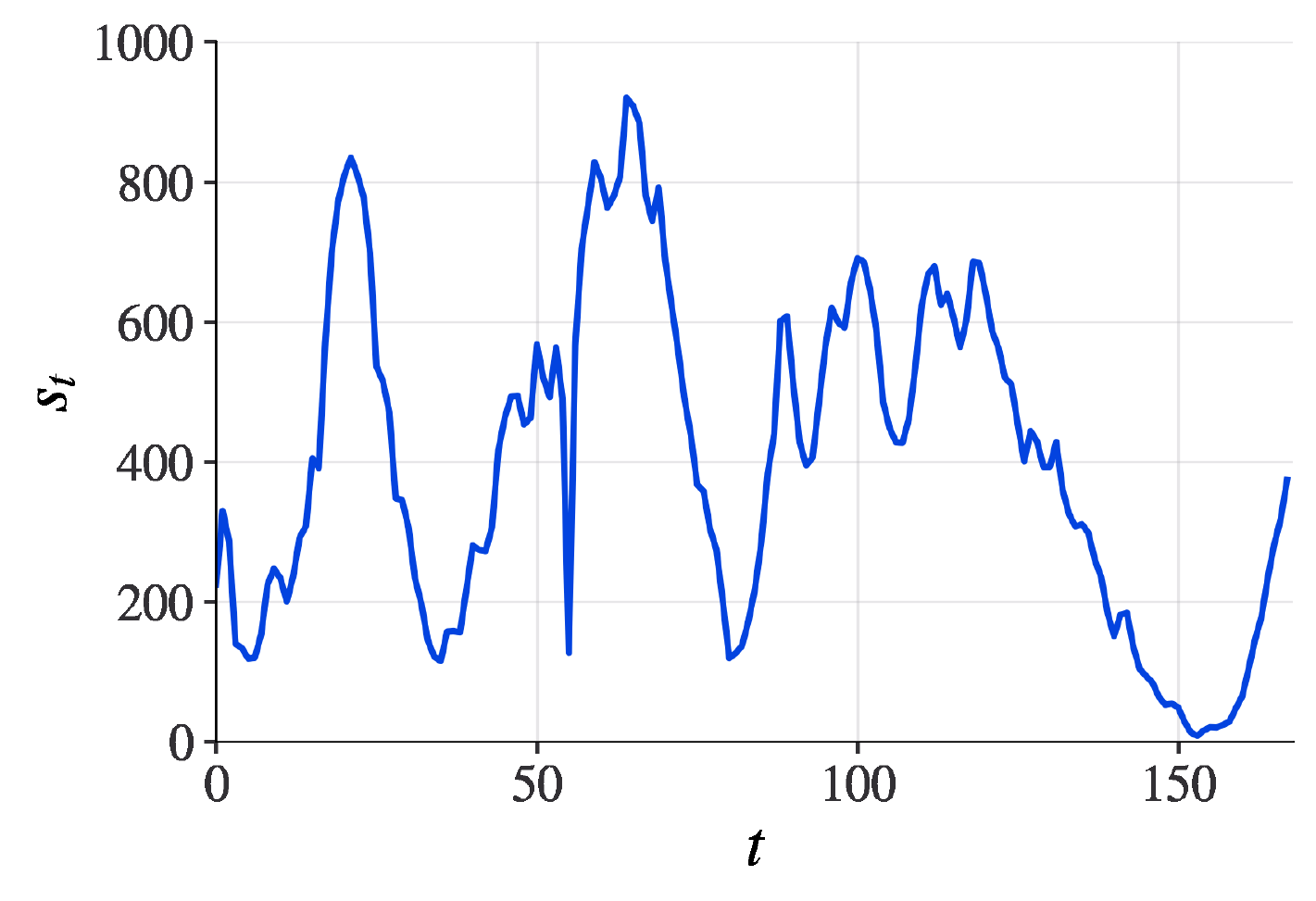}
		\caption{Wind power supply profile, $s_t$}
		\label{fig:pow_dem}
	\end{minipage}
	\begin{minipage}{0.32\textwidth}
		\centering
		\includegraphics[width=\linewidth]{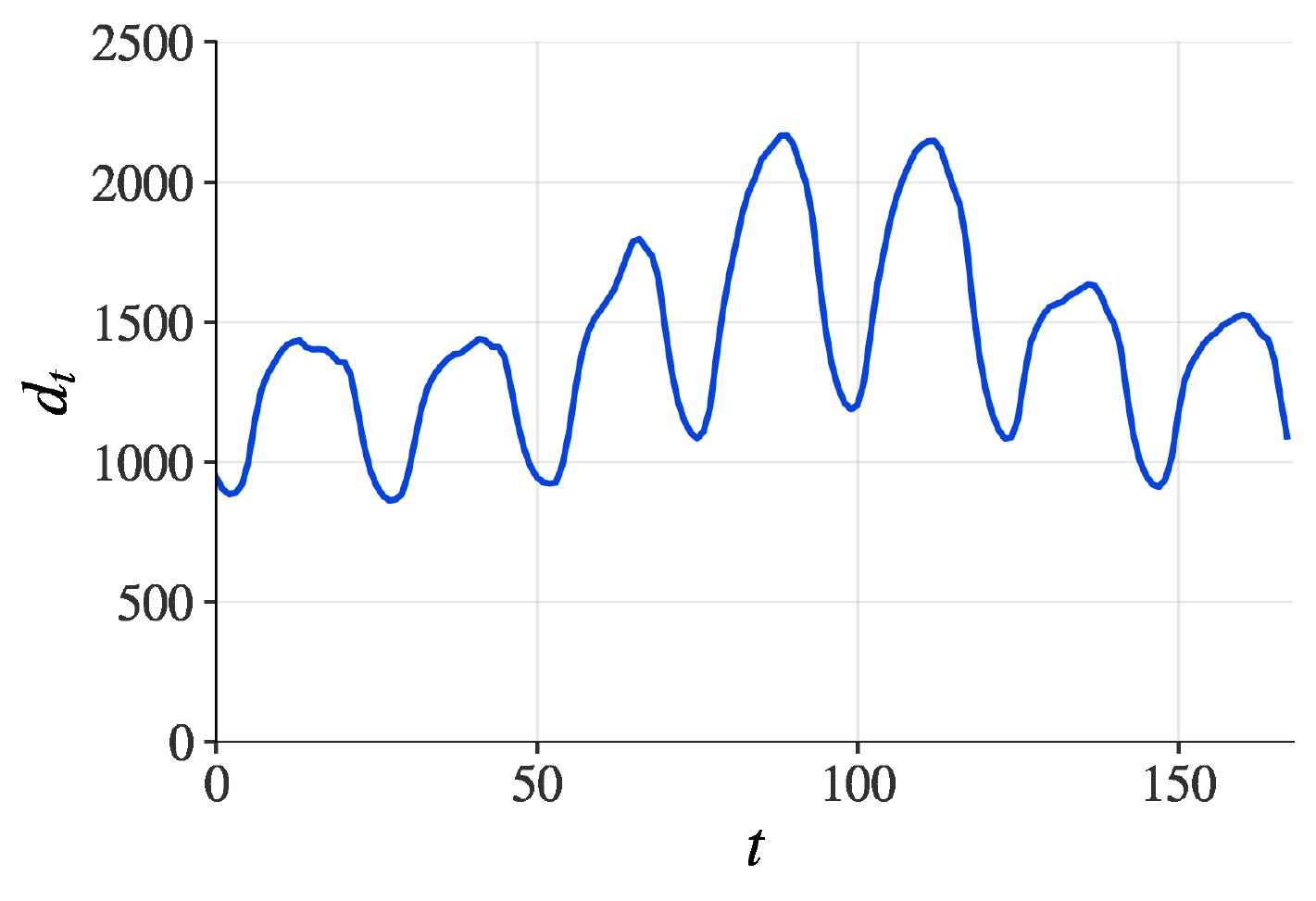}
		\caption{Demand profile, $d_t$}
		\label{fig:pow_sup}
	\end{minipage}
	\begin{minipage}{0.32\textwidth}
		\centering
		\includegraphics[width=\linewidth]{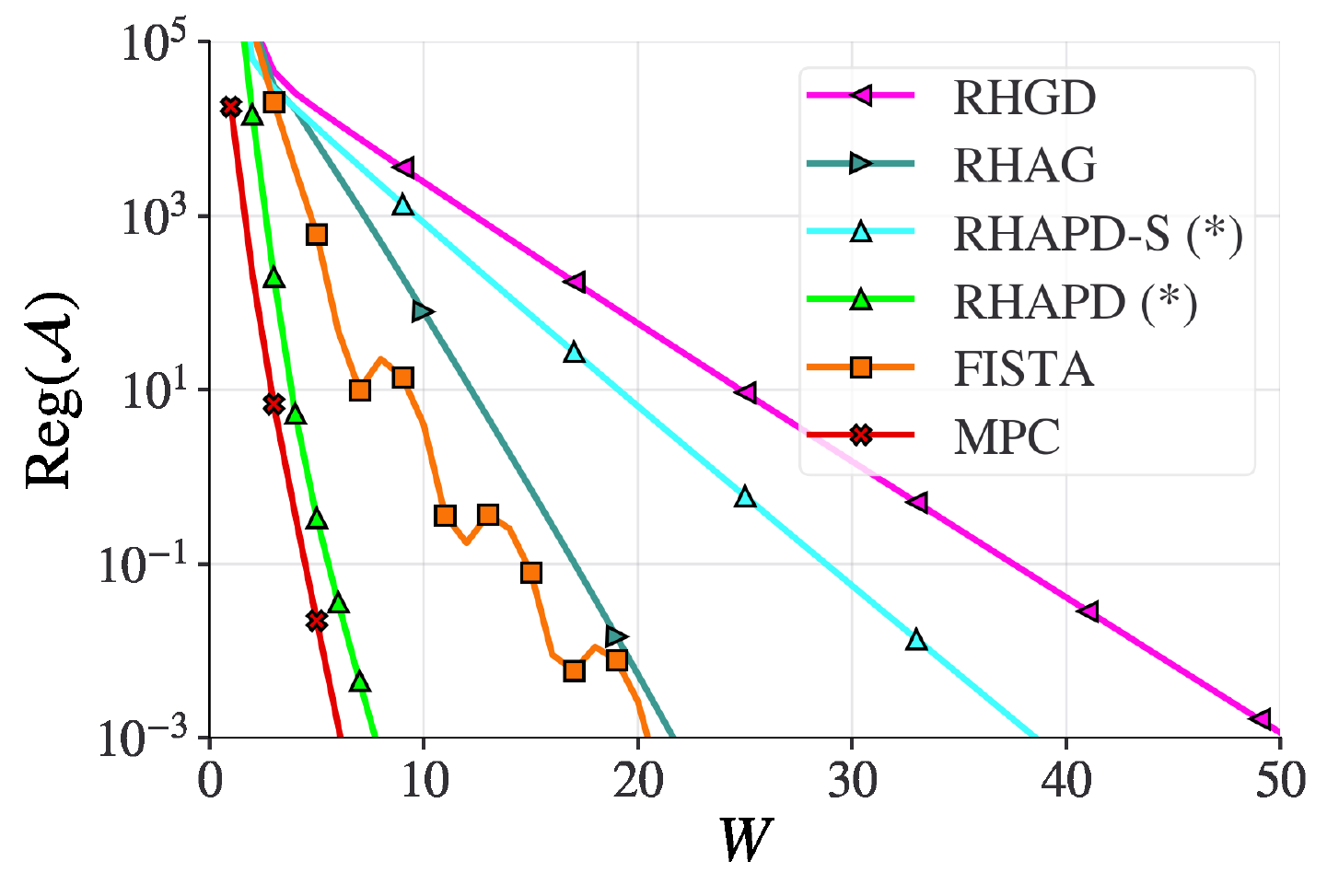}
		\caption{Results for \textbf{E7}}
		\label{fig:pow_res}
	\end{minipage}
\end{figure*}
 \begin{table*}[htb]
	\centering
	\begin{tabular}{|c|c|c|c|c|c|c|c|c|c|c|c|}
	\hline
	Exp & $d$ & $N$ & $\gamma$ & $\cX$ & $\xi_t$ & $x_0$ & $\eta_{OGD}$ & $L_{\nabla H}$ &  $l$ & $\mu$ & $L_{\nabla J}$ \\ \hline
	\textbf{E7}  & $3$  & $168$  & $1$ & $\mathbb{R}_+^3$ & $1.2$ & $\mathbf{0}$ & $1/l$ & $4\gamma$ &  $2\max_t \lambda_{\max}({\mathbf{H}_{f_t}})$   &  $2\min_t \lambda_{\min}({\mathbf{H}_{f_t}})$ & $l+4\gamma$  \\ \hline
	\end{tabular} 
	\caption{
    {Parameters for the economic power dispatch experiment. $\lambda_{\max}(\mathbf{H}_{f_t})$ ($\lambda_{\min}(\mathbf{H}_{f_t})$) denotes the largest (smallest) eigenvalue of the Hessian of $f_t$. } \label{table_ed}
}
\end{table*}
{We set $\gamma = 1$ as the switching cost parameter, and have an imbalance penalty parameter of $\xi_t = \xi = 1.2$. Since generator power outputs cannot be negative, we take the feasible set $\cX = \mathbb{R}_+^3$. For the supply $s_t$, we take the wind generation data for New Hampshire (NH) and for the demand $d_t$, we take the load profile for New Hampshire from the ISO New England operations reports \cite{supply, demand}. These profiles are depicted in Figure \ref{fig:pow_dem} and Figure \ref{fig:pow_sup} respectively. We consider the time period of June 9-15, 2017. Since the data provided is hourly, we have $N=168$ timesteps. Table \ref{table_ed} summarizes the parameters used in this experiment.
}

{The results are presented in Figure \ref{fig:pow_res}. It is interesting to note that RHAPD achieves nearly the same performance as MPC.  As observed in the earlier experiments, RHAPD-S outperforms RHGD. As earlier, the runtime of MPC was again observed to be several times higher than that of the other algorithms. }

\section{Conclusion}

We put forth proximal descent-based algorithms for solving the smoothed online convex optimization (SOCO) problem with predictions. We propose a receding horizon alternating proximal descent (RHAPD) algorithm for proximable stage costs, and a variant RHAPD-S for smooth stage costs. We show that the dynamic regret of the proposed algorithms is upper bounded by a multiple of the path length, and decays exponentially with the length of the prediction window $W$. The bounds can be further tightened when the switching cost is quadratic. Further, we show that the classical alternating minimization algorithm turns out to be a special case of the proposed RHAPD algorithm. We demonstrate the efficacy of our algorithms through numerical experiments on regression, {economic power dispatch}, and trajectory tracking problems. Our algorithms are able to outperform the gradient-based algorithm RHGD and its accelerated variant RHAG, while maintaining the same real-time performance.

Before concluding, we summarize extensions and open problems that remain in this area. First, we note that while the work considered strongly convex stage costs, extension to $p$-uniformly convex functions for $p > 2$ is straightforward. We note that while empirically, the performance of the proposed RHAPD-S algorithm is better than that of RHGD and RHAG, the dynamic regret bounds of RHAPD-S are only better than those of RHGD for large values of switching cost parameter $\gamma$ and always worse than RHAG. The theoretical explanation of this observation, in the form of tighter dynamic regret bounds for RHAPD-S, remains an open problem. 

\bibliographystyle{IEEEtran} 
\bibliography{IEEEabrv,references}
\newpage
\section*{Appendix A (\bf Proximal decrease property)}\label{proxdecproof}
\label{lemma_decrease}
\begin{lemma}
    Let \(s: \cX \to \Rn \cup \{\infty\}\) be a \(\mu\)-strongly convex function and \(h: \cX \to \Rn\) be a $L$-smooth function over some non-empty, closed, and convex set $\cX \subseteq \Rn^d$. Then the proximal update \(\tx = \prox_{\tau s+\ind_\cX} (\x - \tau \nabla h(\x))\) for $\x \in \cX$ and $\tau > 0$ implies that
	\begin{align} \nn
		s(\tx)+h(\tx) \leq s(\x) + h(\x) - \left(\frac{\mu}{2} + \frac{1}{\tau} - \frac{L}{2}\right) \norm{\x - \tx}^{2}.
	\end{align}
\end{lemma}
\begin{proof}
    Let $\u = \x - \tau \nabla h(\x)$ so that $\tx = \prox_{\tau s+\ind_\cX}(\u) = \argmin_{\z \in \cX} s(\z) + \frac{1}{2\tau}\norm{\z-\u}^{2}$. The optimality condition of the proximal operation implies that
	\begin{align}
		\ip{\partial s(\tx)}{\x - \tx} \geq \frac{1}{\tau}\ip{\u-\tx}{\x - \tx}. \label{proxopt}
	\end{align}
 	Next, the $\mu$-strong convexity of $s$ implies the quadratic lower bound on $s(\x)$, which takes the form
 	\begin{align}
 		&s({\x}) \geq s(\tx) + \ip{\partial s(\tx)}{\x - \tx} + \frac{\mu}{2} \norm{\x - \tx}^{2}, \nonumber \\
 		&\geqtext{\eqref{proxopt}} s(\tx)+ \frac{1}{\tau} \ip{\x - \tau \nabla h(\x) - \tx}{\x - \tx} + \frac{\mu}{2} \norm{\x - \tx}^{2}, \nonumber \\
 		& =  s(\tx) + \ip{\nabla h(\x)}{\tx - \x} + \left(\frac{\mu}{2} + \frac{1}{\tau}\right) \norm{\x - \tx}^{2}. \label{qlb}
 	\end{align}
 	Likewise, the $L$-smoothness of $h$ implies the quadratic upper bound on $h(\tx)$ which can be written as
 	\begin{align}
 		h(\x) \geq h(\tx) + \ip{\nabla h(\x)}{\x - \tx} - \frac{L}{2} \norm{\x - \tx}^{2} .\label{qub}
 	\end{align}
 	The desired result follows from adding \eqref{qlb} and \eqref{qub}. 
\end{proof}

\section*{Appendix B (\bf Sufficient decrease property for the iterates of APGD \eqref{alt_prox_update})}\label{proof_of_decrease_non_smooth}
\begin{lemma}\label{decrease_non_smooth}
	Under assumptions \ref{non_smooth}, \ref{g_convex}, the iterates \(\{\x\^k\}\) generated by \eqref{alt_prox_update} satisfy
	\begin{align}\label{eq:suffdecns_supp}
		J(\x\^k) - J(\x\^{k-1}) \leq  - \rho \norm{\x\^k - \x\^{k-1}}^{2},
	\end{align}
	for all $k \geq 1$, where $\rho := \frac{\mu}{2} + \frac{1}{\tau}-l_g$.
\end{lemma}
For the sake of brevity, let $\rho_t := \frac{\mu_t}{2} + \frac{1}{\tau} - l_g$ so that $\rho = \min_{1\leq t \leq N} \rho_t$. Before starting the proof, we first establish the following preliminary result. 
\begin{lemma}\label{hsmooth}
		Under assumption \ref{g_convex}, the function $H_t\^k$ as defined in \eqref{htk} is $2l_g$-smooth over \(\cX\).
\end{lemma}
\begin{proof}
For $1 \leq t \leq N-1$, we have from triangle inequality that
\begin{align}
\norm{\nabla H_t\^k(y) - \nabla H_t\^k(x)}
&\leq \norm{\nabla_1 g(y, x_{t - 1}\^k) - \nabla_1 g(x, x_{t - 1}\^k)} + \norm{\nabla_{2} g(x_{t + 1}\^{k-1}, y) - \nabla_{2} g(x_{t + 1}\^{k-1}, x)}, \nn \\
&\leqtext{\eqref{g_smoothx},\eqref{g_smoothy}} 2l_g\norm{y-x} \nonumber,
\end{align}
for all \(x, y \in \cX\), which implies the $2l_g$-smoothness of $H_t\^k$ over \(\cX\). Since $H_N\^k(x) = g(x, x_{N - 1}\^k)$ is $l_g$-smooth over \(\cX\) \eqref{g_smoothy}, it is also $2l_g$-smooth. 
\end{proof}
We are now ready to prove Lemma \ref{decrease_non_smooth}. Applying Lemma \ref{lemma_decrease} to the update in \eqref{alt_prox_update} with $s(x) = f_t(x)$ (which is \(\mu_t\) strongly convex over \(\cX\)) and $h(x) = H_t\^k(x)$ (which is \(2l_g\) smooth over \(\cX\), as shown in Lemma \ref{hsmooth}), we obtain
\begin{align}
 f_t(x_t\^k) + H_t\^k(x_t\^k) -	f_t(x_t\^{k-1})  - H_t\^k(x_t\^{k-1}) \leq -\rho_t \norm{x_t\^k - x_t\^{k-1}}^{2} \nn,
\end{align}
which upon summing over $t = 1, \ldots, N$, yields
\begin{align}
	\sn f_t(x_t\^k) -\sn f_t(x_t\^{k-1}) + \sn \Big(H_t\^k(x_t\^k)  - H_t\^k(x_t\^{k-1})\Big) &\leq -\sn\rho_t \norm{x_t\^k - x_t\^{k-1}}^{2} \nonumber \\ &\leq  -\rho \norm{\x_t\^k-\x_t\^{k-1}}^2,  \label{fgdiff}
\end{align}
where we have used the fact that $\rho = \min_t \rho_t$. 
Finally, we note that 
\begin{align}
	\sn \left(H_t\^k(x_t\^k)  - H_t\^k(x_t\^{k-1})\right) &= \sn g(x_t\^k, x_{t-1}\^k) + \sum_{t=1}^{N-1}g(x_{t+1}\^{k-1},x_t\^k) \nonumber\\ 
	& \hspace{20mm}- \sn g(x_t\^{k-1}, x_{t-1}\^k)  - \sum_{t=1}^{N-1} g(x_{t+1}\^{k-1},x_t\^{k-1}), \label{telescopic} \\
	=& \sn g(x_t\^k, x_{t-1}\^k) - g(x_1^{(k-1)},x_0^{(k)}) - \sum_{t=2}^{N} g(x_t\^{k-1},x_{t-1}\^{k-1}), \nonumber\\
	= &\sn g(x_t\^k, x_{t-1}\^k) - \sn g(x_t\^{k-1}, x_{t-1}\^{k-1}), \label{gdiff}
\end{align}
where the second and third terms in \eqref{telescopic} add up to yield $-g(x_1^{(k-1)},x_0^{(k)})$, which is then subsumed into the last term since $x_0\^k = x_0 = x_0\^{k-1}$. Substituting \eqref{gdiff} into \eqref{fgdiff}, we obtain the required result. This completes the proof. It is remarked that for the bound to be useful, we require $\tau$ to be such that $\rho > 0$. 

\section*{Appendix C (\textbf{Subgradient bound for the iterates of APGD \eqref{alt_prox_update}})}\label{proof_of_subgrad_bound_non_smooth}
\begin{lemma}\label{subgrad_bound_non_smooth}
	Under assumption \ref{g_convex}, for all $k \geq 1$, there exists a $\v\^k \in \partial \Jt(\x\^k)$, such that 
	\begin{align}
		\norm{\v\^k} \leq \beta \norm{\x\^k - \x\^{k-1}}, \label{subgradnsboundeqn}
	\end{align}
	 where $\beta^2 := 2\left(\sqrt{5}l_g + \frac{1}{\tau}\right)^2$.
	\end{lemma}
\begin{proof}
    The proof follows from the smoothness of the penalties $g$. Let $[\partial \Jt(\x)]_t$ denote the partial subgrad-differential of $\Jt(\x)$ with respect to $x_t$. We split the proof into two cases, that of $1 \leq t \leq N-1$ and that of $t = N$. 

\noindent\subsubsection*{Case $1 \leq t\leq N-1$} From the definition of $\Jt$ in \eqref{jtdef}, we have
\begin{align}
	[\partial \Jt(\x\^k)]_t = \partial f_t(x_t\^k)+ \nabla_1 g(x_t\^k, x_{t - 1}\^k) + \nabla_{2} g(x_{t + 1}\^k, x_t\^k) + \partial \ind_{\cX}(x_t\^k). \label{subjt}
\end{align}
The optimality condition of \eqref{alt_prox_update} can be written as 
	 \begin{align}
		\mathbf{0} \in \partial f_t(x_t\^k) + \partial \ind_{\cX}(x_t\^k) + \frac{x_t\^k - x_t\^{k-1}}{\tau} \nonumber + \nabla_1 g(x_t\^{k-1}, x_{t - 1}\^k) + \nabla_{2} g(x_{t + 1}\^{k-1}, x_t\^{k-1}) \nn.
	\end{align}
	Substituting the definition of $[\partial \Jt(\x\^k)]_t$ in \eqref{subjt}, it follows that $v_t\in [\partial \Jt(\x\^k)]_t$ where
	\begin{align}
			v_t\^k := \frac{x_t\^{k-1} - x_t\^k}{\tau} - \nabla_1 g(x_t\^{k-1}, x_{t - 1}\^k) - \nabla_{2}g(x_{t + 1}\^{k-1}, x_t\^{k-1}) + \nabla_1 g(x_t\^k, x_{t - 1}\^k) + \nabla_{2} g(x_{t + 1}\^k, x_t\^k).\label{jtisubgrad}
	\end{align}
	Using triangle inequality, $\norm{v_t\^k}$ can be bounded by
	\begin{align}\label{vtbound}
		\norm{v_t\^k} \leq \frac{1}{\tau}\norm{x_t\^{k-1} - x_t\^k} &+  \norm{\nabla_1 g(x_t\^k, x_{t - 1}\^k) - \nabla_1 g(x_t\^{k-1}, x_{t - 1}\^k)} \\ &+ \norm{\nabla_{2} g(x_{t + 1}\^k, x_t\^k) - \nabla_{2}g(x_{t + 1}\^{k-1}, x_t\^{k-1})}.
	\end{align}
	We next use assumption \ref{g_convex} to bound the second and third terms in \eqref{vtbound}. For the second term, we have that
	\begin{align}
		\norm{\nabla_1 g(x_t\^k, x_{t - 1}\^k) - \nabla_1 g(x_t\^{k-1}, x_{t - 1}\^k)} \leqtext{\eqref{g_smoothx}} l_g\norm{x_t\^k-x_t\^{k-1}}, \label{secondterm}
	\end{align} 
	while for the third term, we have
	\begin{align}
		\norm{\nabla_{2} g(x_{t + 1}\^k, x_t\^k) - \nabla_{2}g(x_{t + 1}\^{k-1}, x_t\^{k-1})}
		&\leq \norm{\nabla_{2} g(x_{t + 1}\^k, x_t\^k) - \nabla_{2}g(x_{t + 1}\^k, x_t\^{k-1})} \nonumber \\ 
		&\hspace{5mm}+ \norm{\nabla_{2}g(x_{t + 1}\^k, x_t\^{k-1}) - \nabla_{2}g(x_{t + 1}\^{k-1}, x_t\^{k-1})},\nonumber\\
		&\leqtext{\eqref{g_smoothy}} l_g\norm{x_t\^k-x_t\^{k-1}} + l_g\norm{x_{t+1}\^k - x_{t+1}\^{k-1}}, \label{thirdterm}
	\end{align}
	where in \eqref{thirdterm}, we have also used \eqref{g_smooth}, which implies that \[\norm{\nabla_2g(x_{t + 1}\^k, x_t^{(k - 1)}) - \nabla_2g(x_{t + 1}\^{k - 1},x_t\^{k - 1})} \leq l_g \norm{x_{t + 1}\^k - x_{t + 1}\^{k - 1}}.\] 
	Substituting \eqref{secondterm} and \eqref{thirdterm} into \eqref{vtbound}, we obtain
	\begin{align}
		\norm{v_t\^k} &\leq \left(2l_g + \frac{1}{\tau}\right)\norm{x_t\^k-x_t\^{k-1}} +  l_g\norm{x_{t+1}\^k - x_{t+1}\^{k-1}}, \nonumber 
  \end{align}
  which implies 
  \begin{align}
		\norm{v_t\^k}^2 &\leq 2\left(2l_g + \frac{1}{\tau}\right)^2\norm{x_t\^k-x_t\^{k-1}}^2 +  2l_g^2\norm{x_{t+1}\^k - x_{t+1}\^{k-1}}^2, \label{case1}
	\end{align}
where we have used the inequality \((a + b) ^ {2} \le 2 a ^ 2 + 2 b ^ 2\) which holds for all \(a, b \in \Rn\).
	\subsubsection*{Case $t = N$} In this case, we have from \eqref{jtdef} that
	\begin{align}
		[\partial \Jt(\x\^k)]_N = \partial f_N(x_N\^k) + \nabla_1g(x_N\^k, x_{N - 1}\^k) + \partial \ind_{\cX}(x_N\^k), \label{subjtn}
	\end{align}
	while the optimality condition of \eqref{alt_prox_update} is written as
	\begin{align}\nn
		\mathbf{0}  \in \partial f_N(x_N\^k) &+ \partial \ind_{\cX}(x_N\^k) + \frac{x_N\^k - x_N\^{k-1}}{\tau} + \nabla_1 g(x_N\^{k-1}, x_{N - 1}\^k).
	\end{align}
	Substituting the definition of $[\partial \Jt(\x\^k)]_N$ in \eqref{subjtn}, we have that ${v_N ^ {(k)}} \in [\partial \Jt(\x\^k)]_N$ where 
	\begin{align}
		v_N\^k := \frac{x_N\^{k-1} - x_N\^k}{\tau} +  \nabla_1 &g(x_N\^k, x_{N - 1}\^k) - \nabla_1 g(x_N\^{k-1}, x_{N - 1}\^k), \nn
	\end{align}
	whose norm can be bounded using the triangle inequality and assumption \ref{g_convex} as
	\begin{align}
		\norm{v_N\^k} &\leq \frac{1}{\tau}\norm{x_N\^{k-1} - x_N\^k} + \norm{\nabla_1 g(x_N\^k, x_{N - 1}\^k) - \nabla_1 g(x_N\^{k-1}, x_{N - 1}\^k) },\nonumber \\
		&\leqtext{\eqref{g_smoothx}} \left(l_g+\frac{1}{\tau}\right)\norm{x_N\^{k-1} - x_N\^k}, \label{case2}
	\end{align}
	Combining the two cases, it can be seen that there exists $\v \in \partial \Jt(\x\^k)$ such that 
	\begin{align}
		\norm{\v\^k}^2 &\leqtext{\eqref{case1},\eqref{case2}} 2\sum_{t=1}^{N-1} \left(2l_g + \frac{1}{\tau}\right)^2\norm{x_t\^k-x_t\^{k-1}}^2 +  2\sum_{t=1}^{N-1} l_g^2\norm{x_{t+1}\^k - x_{t+1}\^{k-1}}^2 + \left(l_g+\frac{1}{\tau}\right)^2\norm{x_N\^{k-1} - x_N\^k}^2, \nn\\
		&\leq 2\left(\sqrt{5}l_g + \frac{1}{\tau}\right)^2\sn \norm{x_t\^k-x_t\^{k-1}}^2 \nonumber = 2\left(\sqrt{5}l_g + \frac{1}{\tau}\right)^2 \norm{\x\^k-\x\^{k-1}}^2. \nn
	\end{align}
\end{proof}

\section*{Appendix D (\textbf{Regret of initialization for RHAPD})} \label{proof_dynamic_regret_non_smooth_general_g}
\begin{lemma}\label{dynamic_regret_non_smooth_initialization}
		Under assumptions \ref{bounded_subgradient}, \ref{non_smooth}, \ref{g_convex}, the regret attained by the intialization policy \(\mathcal{I}\), that is \(x_{t + W} ^ {(0)} = \argmin_{x \in \mathcal{X}} f_{t + W - 1}(x)\), is bounded by \begin{equation}\nn
			\mathrm{Reg}(\mathcal{I}) \leq G\bigg(1 + \frac{\gamma}{\mu} \bigg) \left(\sum_{t = 1} ^ {N} \norm{\theta_t - \theta_{t - 1}}\right),
		\end{equation}
		where \(\theta_0 = x_0, \theta_t = \argmin_{x \in \mathcal{X}} f_t(x)\); \(\rho\) and \(\beta\) are the constants as defined in Lemma \ref{decrease_non_smooth} and Lemma \ref{subgrad_bound_non_smooth} respectively.
	\end{lemma}
	\begin{proof}
	   The policy initializes \(x_i\^0 = \theta_{i - 1}\) for all \(1 \leq i \leq N\), and for the sake of brevity we assume \(\theta_{0} = x_{0}\). It thus follows that \begin{align} \mathrm{Reg}(\mathcal{I}) &= \sum_{t = 1}^{N} \bigg(f_t(x_t\^0) + g(x_{t} ^ {(0)}, x_{t - 1} ^ {(0)}) - f_t(x_t^\star) - g(x_{t} ^ {*}, x_{t - 1} ^ {*})\bigg),\nonumber \\
		& \stackrel{(a)}\leq \sum_{t = 1}^{N} f_t(x_t\^0) - f_t(\theta_t) + g(x_t\^0, x_{t - 1}\^0) ,\nonumber \\
		& \leqtext{\eqref{g_convex}} \sum_{t = 1} ^ {N} f(x_t\^0) - f_t(\theta_t) + \frac{\gamma}{2} \norm{x_t\^0 - x_{t - 1}\^0} ^ 2, \nonumber \\
		&\stackrel{(b)}\leq G \sum_{t = 1}^{N}\norm{x_t\^0 - \theta_t} + \sum_{t = 1}^{N - 1}\frac{\gamma}{2} \norm{\theta_{t} - \theta_{t - 1}}^{2}, \nn
  \end{align}	
  where \((a)\) follows since \(g(x, y) \ge 0\) and \(f_t({\theta_{t}}) \le f_{t}(x_{t}^\star)\), while \((b)\) follows from assumption \ref{bounded_subgradient} which implies \(f_{t}(x_{t} ^ {(0)}) - f_{t}(\theta_t) \le G \norm{x_{t} ^ {(0)} - \theta_{t}}\).

It remains to bound \(\sum_{t = 1}^{N - 1} \norm{\theta_t - \theta_{t - 1}}^{2}\). For \(1 \leq t \leq N\), the strong convexity of \(f_t(x)\) over \(\cX\) and \(\theta_t = \argmin_{x \in \cX} f_t(x)\) implies 
\begin{align}\nn f_t(x) \geq f_t(\theta_t) + \frac{\mu_t}{2} \norm{x - \theta_t}^{2},\end{align} for all \(x \in \cX\). This result follows from \cite[Equation 2]{hazan2011beyond}.
Setting \(x = \theta_{t - 1}\), we get
\begin{align}\nn f_t(\theta_{t - 1}) - f_t(\theta_t) \geq \frac{\mu_t}{2} \norm{\theta_{t - 1} - \theta_t}^{2}.\end{align}
Using assumption \ref{bounded_subgradient} as before, we obtain
\begin{align} \frac{2G}{\mu_t} \norm{\theta_{t - 1} - \theta_t} \geq \norm{\theta_t - \theta_{t - 1}}^{2} \implies \sum_{t = 1}^{N} \norm{\theta_t - \theta_{t-  1}}^{2} \leq \frac{2G}{\mu} \sum_{t = 1}^{N} \norm{\theta_{t - 1} - \theta_t}. \nn \end{align} 
The regret of the policy \(\mathcal{I}\) is therefore bounded as 
\begin{align}\nn
	\mathrm{Reg}(\mathcal{I}) \leq G\left(1 + \frac{\gamma}{\mu}\right) \sum_{t = 1}^{N} \norm{\theta_t - \theta_{t - 1}}.
\end{align} 
\end{proof}
\section*{Appendix E (\bf Sufficient decrease property and subgradient bound for APGD with quadratic switching cost)} \label{proof_subgrad_bound_quadratic}
We begin with stating the sufficient decrease property for the quadratic switching cost, which is a direct implication of Lemma \ref{decrease_non_smooth} and the fact that $g(x,y) = \frac{\gamma}{2}\norm{x-y}^2$ is $\gamma$-smooth with respect to both $x$ and $y$. 
\begin{corollary}\label{decrease_quadratic} Under assumption \ref{non_smooth} and for $g(x,y) = \frac{\gamma}{2}\norm{x-y}^2$, the iterates $\{\x\^k\}$ generated by \eqref{alt_prox_update} satisfy \begin{align}\nn
		J(\x\^k)-J(\x\^{k-1}) \leq -{\rho_q}\norm{\x\^k-\x\^{k-1}}^2,
	\end{align}
	for all $k \geq 1$ and ${\rho_q} := \frac{\mu}{2} + \frac{1}{\tau} - \gamma$.
\end{corollary}
The proof of Corollary \ref{decrease_quadratic} is very similar to Lemma \ref{decrease_non_smooth} and is skipped for the sake of brevity.
\begin{lemma}\label{subgrad_bound_quadratic}
	For all \(k \geqslant 1\), there exists $\v\^k \in \partial \Jt(\x\^k)$ such that 
	\begin{align}\nn
		\norm{\v\^k} \leq \beta_q \norm{\x\^k-\x\^{k-1}},
	\end{align}
	 where 
	\begin{align}\nn
		\beta^2_q := 2\left(\gamma^2 + \max\left\{\left(2\gamma - \frac{1}{\tau}\right)^2, \left(\gamma - \frac{1}{\tau}\right)^2\right\}\right).
	\end{align}
	\end{lemma}
\begin{proof}
    The proof of {Lemma} \ref{subgrad_bound_quadratic} differs from that of Lemma \ref{subgrad_bound_non_smooth} since the specific form of $g$ allows for tighter bounds. As in Appendix \hyperref[proof_of_subgrad_bound_non_smooth]{C}, we split the proof into two cases. 

\subsubsection*{Case $1 \leq t \leq N - 1$} 
Considering $g(x,y) = \frac{\gamma}{2}\norm{x-y}^2$, we have from \eqref{jtisubgrad} that $v_t \in [\partial \Jt(\x\^k)]_t$ where
\begin{align}
	v_t\^k &:= \frac{x_t\^{k-1} - x_t\^k}{\tau} - \gamma(x_t\^{k-1} - x_{t - 1}\^k) - \gamma(x_t\^{k-1} - x_{t + 1}\^{k-1} ) + \gamma(x_t\^k -  x_{t - 1}\^k)+ \gamma(x_t\^k-x_{t + 1}\^k), \nonumber\\
	&= \left(2\gamma - \frac{1}{\tau}\right)(x_t\^k - x_t\^{k-1}) + \gamma(x_{t + 1}\^{k-1} - x_{t + 1}\^k).\nonumber
\end{align}
A bound on $\norm{v_t\^k}^2$ may therefore be obtained as in Appendix \hyperref[proof_of_subgrad_bound_non_smooth]{C} and takes the form
\begin{align}
	\norm{v_t\^k}^2 &\leq 2\left(2\gamma - \frac{1}{\tau}\right)^2\norm{x_t\^k - x_t\^{k-1}}^2 + 2\gamma^2\norm{x_{t + 1}\^{k-1} - x_{t + 1}\^k}^2. \label{vtbound-q}
\end{align}
	
\subsubsection*{Case $t=N$}
Proceeding in the same way, we have that $	v_N\^k := \left(\gamma - \frac{1}{\tau}\right)(x\^k_N - x\^{k-1}_N) \in [\partial \Jt(\x\^k)]_N$, which can be bounded as 
\begin{align}
	\norm{{v_{N} ^ {(k)}}}^2 \leq \left(\gamma - \frac{1}{\tau}\right)^2\norm{x_N\^k - x_N\^{k-1}} ^ 2. \label{vnbound-q}
\end{align}

Combining the two cases, we see that there exists $\v\^k \in \partial \Jt(\x\^k)$ such that 
\begin{align}
	&\norm{\v\^k}^2 = \sn \norm{v_t\^k}^2,  \nonumber \\ 
	&\leqtext{\eqref{vtbound-q}, \eqref{vnbound-q}} 2\left(2\gamma - \frac{1}{\tau}\right)^2 \norm{x_1\^k - x_1\^{k-1}}^{2} + 2\left(\gamma^2 + \left(\gamma - \frac{1}{\tau}\right)^2\right) \norm{x_N\^k - x_N\^{k-1}}^{2} \nonumber\\
	&\hspace{60mm}+  \sum_{t = 2} ^ {N - 1} 2 \left( \gamma^2 + \left(2\gamma - \frac{1}{\tau} \right)^2\right) \norm{x_t\^k - x_t\^{k-1}}^2,\nn
\end{align}
which implies the required bound in Lemma \ref{subgrad_bound_quadratic}. 
\end{proof}

\section*{Appendix F (Sufficient decrease property for the iterates of APGD-S \eqref{xtk-update-smooth})}\label{proof_of_sufficient_decrease_smooth}
\begin{lemma} \label{lem:sufficient_decrease_smooth}
	Under assumption \ref{smooth}, the iterates $\x\^k$ generated by \eqref{xtk-update-smooth} satisfy
	\begin{align}
		J(\x\^k) - J(\x\^{k-1}) \leq - \rho_s \norm{\x\^k-\x\^{k-1}}^2, \nn
	\end{align}
	for all $k \geq 1$ and $\rho_s := \frac{1}{\tau} - \frac{l}{2} + \frac{\gamma}{2}$. 
\end{lemma}
\begin{proof}
    We observe that $s(x) = H_t\^k(x)$ is $2\gamma$-strongly convex over \(\cX\) for \(1 \leq t \leq N - 1\) and \(\gamma\)-strongly convex over \(\cX\) for \(t = N\). Further, $h(x) = f_t(x)$ is $l_t$-smooth over \(\cX\) for all $1\leq t \leq N$. Applying Lemma \ref{lemma_decrease} to the update in \eqref{xtk-update-smooth}, we obtain the following for \(1 \leq t \leq N - 1\)
\begin{align}
	H_t(x_t\^k) + f_t(x_t\^k) &\leq H_t(x_t\^{k-1}) + f_t(x_t\^{k-1}) - \left(\gamma + \frac{1}{\tau} - \frac{l_t}{2}\right) \norm{x_t\^k - x_t\^{k-1}}^{2}, \nonumber
\end{align}
and for \(t = N\), we have
\begin{align}
	H_N(x_N\^k) + f_N(x_N\^k) &\leq   H_N(x_N\^{k-1}) + f_N(x_N\^{k-1}) - \left(\frac{\gamma}{2} + \frac{1}{\tau} - \frac{l_N}{2}\right) \norm{x_N\^k - x_N\^{k-1}}^2.  \nonumber
\end{align}
Summing over $t = 1, \ldots, N$ and following steps \ref{telescopic}-\ref{gdiff} as before, we get the required bound.
\end{proof}

\section*{Appendix G (Subgradient bound for the iterates of APGD-S \eqref{xtk-update-smooth})}\label{proof_of_subgrad_bound_smooth}
\begin{lemma}\label{subgrad_bound_smooth}
  	Under assumption \ref{smooth}, for all \(k \geq 1\), there exists $\v\^k \in \partial \Jt(\x\^k)$, such that
  	\begin{align}\nn
  		\norm{\v\^k} \leq \beta \norm{\x\^k-\x\^{k-1}},
  	\end{align}
  	where $\beta_s^2 := 2 \left(l + \gamma + \frac{1}{\tau}\right)^2$. 
   \end{lemma}
\begin{proof}
    The proof follows from the smoothness of the stage costs $f_t$. As in Appendix \hyperref[proof_of_subgrad_bound_non_smooth]{C}, we split the proof into two cases, depending on the value of $t$.

\subsubsection*{Case $1 \leq t \leq N-1$} The optimality condition of \eqref{xtk-update-smooth} can be written as
\begin{align}
	\mathbf{0} &\in \gamma \Big(2x_t\^k - x_{t - 1}\^k - x_{t + 1}\^{k-1}\Big) + \frac{x_t\^k - x_t\^{k-1}}{\tau} + \nabla f_t(x_t\^{k-1}) + \partial \ind_{\cX}(x_t\^k).\nn
\end{align}
Recalling the definition of $[\partial \Jt(\x\^k)]_t$ from \eqref{subjt} for the quadratic switching cost, we have that $v_t\^k \in [\partial \Jt(\x\^k)]_t$ where 
\begin{align}
	v_t\^k &:= \nabla f_t(x_t\^k) - \nabla f_t(x_t\^{k-1}) + \frac{x_t\^{k-1} - x_t\^k}{\tau}  + \gamma(x_{t + 1}\^{k-1} - x_{t + 1}\^k)\nn.
\end{align}
Using the triangle inequality and the $l_t$-smoothness of $f_t$, we can bound $\norm{v_t\^k}$ as
\begin{align}
	\norm{v_t\^k} &\leq  \Big\lVert \nabla f_t(x_t\^k) - \nabla f_t(x_t\^{k-1})\Big\rVert \nn + \frac{1}{\tau}\norm{x_t\^{k-1} - x_t\^k} + \gamma\norm{x_{t + 1}\^{k-1} - x_{t + 1}\^k} , \nonumber \\
	&\leqtext{\eqref{smooth}}   \left( l_t + \frac{1}{\tau} \right) \norm{x_t\^k - x_t\^{k-1}} + \gamma \norm{x_{t + 1}\^{k-1} - x_{t + 1}\^k}, \nonumber
\end{align}
which implies that 
\begin{align}
\norm{v_t\^k}^2 &\leq  2\left( l_t + \frac{1}{\tau} \right)^2 \norm{x_t\^k - x_t\^{k-1}}^2 + 2\gamma^2 \norm{x_{t + 1}\^{k-1} - x_{t + 1}\^k}^2. \label{vtbound-s}
\end{align}

\subsubsection*{Case $t=N$} For this case, the optimality condition of \eqref{xtk-update-smooth} is given by 
\begin{align}\nn
	\mathbf{0} \in \gamma(x_N\^k - x_{N - 1}\^k) + \frac{x_N\^k - x_N\^{k-1}}{\tau}+ \nabla f_N(x_N\^{k-1}) + \partial \ind_{\cX}(x_N\^k).
\end{align}
From the definition of $[\partial \Jt(\x\^k)]_N$ in \eqref{subjtn} for the quadratic switching cost, we infer that $v_N\^k \in [\partial \Jt(\x\^k)]_N$, where
\begin{align}
	v_N\^k := \nabla f_N(x_N\^k) - \nabla f_N(x_N\^{k-1}) + \frac{x_N\^{k-1} - x_N\^k}{\tau}. \nn
\end{align}
The norm of $v_N\^k$ can be bounded using the triangle inequality and the $l_N$-smoothness of $f_N$ and takes the form
\begin{align}
	\norm{v_N\^k} \leq \Big\lVert \nabla f_N(x_N\^k) - \nabla f_N(x_N\^{k-1}) \Big\rVert + \frac{1}{\tau}\norm{x_N\^{k-1} - x_N\^k}\leqtext{\eqref{smooth}}  \left(l_N + \frac{1}{\tau} \right) \norm{x_N\^k - x_N\^{k-1}} .\label{vnbound-s}
\end{align}
Combining \eqref{vtbound-s}, \eqref{vnbound-s}, we obtain
\begin{align}
	\norm{\v\^k}^2 &\leqtext{\eqref{vtbound-s},\eqref{vnbound-s}} 2\sum_{t=1}^{N-1}\left(l_t + \frac{1}{\tau}\right)^2\norm{x_t\^k-x_t\^{k-1}}^2 +  2\sum_{t=1}^{N-1} \gamma^2\norm{x_{t+1}\^k - x_{t+1}\^{k-1}}^2 + \left(l_N+\frac{1}{\tau}\right)^2\norm{x_N\^{k-1} - x_N\^k}^2, \nonumber 
 \end{align}
 which implies \(\norm{\v\^k}^2 \leq  \sn \beta_t^2\norm{x_t\^k-x_t\^{k-1}}^2\leq \beta_s^2 \norm{\x\^k-\x\^{k-1}}^2\),
where $\beta_t^2 = 2\left(l_t + \gamma + \frac{1}{\tau}\right)^2$. Defining $\beta_s = \max_t  \beta_t$, we obtain the required result.
\end{proof}
\section*{Appendix H (Smoothness parameter of the sum-squared switching cost)}\label{app:smooth}
We show that the function \(g(x, y) = \frac{\gamma}{2\sqrt{2}d} \ip{x - y}{\mathbf{1}} ^ 2\) is \(\gamma\)-smooth over \(\Rn ^ d \times \Rn ^ d\).
We have \(g(x, y) = \frac{\gamma}{2\sqrt{2}d} (\sum_{i = 1} ^ {d} x_{i} - \sum_{i = 1} ^ d y_{i}) ^ {2}\). Therefore, \begin{align*}
    \nabla_{1} g(x, y) = \frac{\gamma}{\sqrt{2}d} \big(\sum_{i = 1} ^ d (x_{i} - y_{i})\big) \mathbf{1}, \quad \nabla_{2} g(x, y) = -\frac{\gamma}{\sqrt{2}d} \big(\sum_{i = 1} ^ d (x_{i} - y_{i})\big) \mathbf{1}.
\end{align*}
Hence, we obtain \begin{align}\label{eq:bound_this}
	\norm{\begin{bmatrix} \nabla_1 g(u, v) \\ \nabla_2 g(u, v)\end{bmatrix} - \begin{bmatrix} \nabla_1 g(x, y) \\ \nabla_2 g(x, y)\end{bmatrix}} = \frac{\gamma}{\sqrt{2}d} \norm{\sum_{i = 1} ^ {d} (u_{i} - v_{i}) w - \sum_{i = 1} ^ d (x_{i} - y_{i}) w},
\end{align}
where \(\w := [\underbrace{1, \dots, 1}_{d}, \underbrace{-1, \dots, -1}_{d}] ^ {\intercal} \in \Rn ^ {2d}\). Bounding the quantity in \eqref{eq:bound_this}, we obtain \begin{align*}
    \norm{\begin{bmatrix} \nabla_1 g(u, v) \\ \nabla_2 g(u, v)\end{bmatrix} - \begin{bmatrix} \nabla_1 g(x, y) \\ \nabla_2 g(x, y)\end{bmatrix}} &= \frac{\gamma}{\sqrt{2d}} \abs{\sum_{i = 1} ^ {d} (u_{i} - v_{i}) + \sum_{i = 1} ^ d (y_{i} - x_{i})}, \\
    &= \frac{\gamma}{\sqrt{2d}} \abs{\sum_{i = 1} ^ {d} (u_{i} - x_{i}) + \sum_{i = 1} ^ d (y_{i} - v_{i})}, \\
    &\stackrel{(a)}\le \gamma \sqrt{\sum_{i = 1} ^ {d} (u_{i} - x_{i}) ^ 2 + \sum_{i = 1} ^ d (y_{i} - v_{i}) ^ 2}, \\
    &= \gamma \norm{\begin{bmatrix} u \\ v \end{bmatrix} - \begin{bmatrix} x \\ y \end{bmatrix}},
\end{align*}
where \((a)\) follows from the inequality \(\big(\sum_{i = 1} ^ {k} a_{i}\big) ^ 2 \le k \sum_{i = 1} ^ k a_{i} ^ 2\). This completes the proof.

\section*{Appendix I (Smoothness parameter of \(g({\x}) = \sum_{i = 1} ^ {N} \frac{\gamma}{2}\ip{x_{i} - x_{i - 1}}{\mathbf{1}} ^ 2\))}\label{app:overall_smooth}
Here, we show that \(g({\x})\) is \(4\gamma d\)-smooth over \(\Rn ^ {d} \times \dots \times \Rn ^d\). The gradient \(\nabla_{x_{i}} g(\x)\) can be expressed as \begin{align*}
    \nabla_{x_{i}} g(\x) = \begin{cases}
    \gamma\ip{2x_{i} - x_{i - 1} - x_{i + 1}}{\mathbf{1}} \mathbf{1} & 1 \leq i \leq N - 1, \\
    \gamma\ip{x_{i} - x_{i - 1}}{\mathbf{1}}\mathbf{1} & i = N.
    \end{cases}
\end{align*}
Therefore, we can bound \(\norm{\nabla_{\x} g(\x) - \nabla_{\y} g(\y)} ^ 2\) in the following manner:
\begin{align}
    \norm{\nabla_{\x} g(\x) - \nabla_{\y} g(\y)} ^ 2 &= \sum_{i = 1} ^ {N} {\norm{\nabla_{x_{i}} g(\x) - \nabla_{y_{i}} g(\y)} ^ 2}, \nn \\
    &= \gamma ^ 2 d\bigg(\ip{2(x_{1} - y_{1}) - (x_{2} - y_{2})}{\mathbf{1}} ^ 2  \nonumber \\ &\hspace{20mm} + \sum_{i = 2} ^ {N - 1} {\ip{2(x_{i} - y_{i}) - (x_{i - 1} - y_{i - 1}) - (x_{i + 1} - y_{i + 1})}{\mathbf{1}}} ^ 2\nn \nonumber \\
    &\hspace{20mm} + {\ip{(x_{N} - y_{N}) - (x_{N - 1} - y_{N - 1})}{\mathbf{1}}} ^ 2\bigg),\nn \\
    &\stackrel{(a)}\leq \gamma ^ 2 d ^ 2\bigg(\norm{2(x_{1} - y_{1}) + (x_{2} - y_{2})} ^ 2 \nonumber \\ &\hspace{20mm} +  \sum_{i = 2} ^ {N - 1} \norm{2(x_{i} - y_{i}) - (x_{i - 1} - y_{i - 1}) - (x_{i + 1} - y_{i + 1})} ^ {2} \nn\\
    & \hspace{20mm} + \norm{(x_{N} - y_{N}) - (x_{N - 1} - y_{N - 1})} ^ {2} \bigg), \label{eq:continue_from_here}
\end{align}
where \((a)\) follows from the inequality \(\big(\sum_{i = 1} ^ {d} a_{i}\big) ^ 2 \le d \sum_{i = 1} ^ d a_{i} ^ 2\).
Next, we consider the function \(\phi(\x)\) defined as \begin{align*}
    \phi(\x) := \sum_{i = 1} ^ {N} \frac{\gamma}{2} \norm{x_{i} - x_{i - 1}} ^ {2},
\end{align*}
The gradient \(\nabla_{x_{i}} \phi(\x)\) can be expressed as \begin{align}\label{eq:nabla_phi}
    \nabla_{x_{i}} \phi(\x) = \begin{cases}
    \gamma (2x_{i} - x_{i - 1} - x_{i + 1}) & 1 \leq i \leq N - 1, \\
    \gamma (x_{i} - x_{i - 1}) & i = N.
    \end{cases}
\end{align}
It follows from \cite[Lemma 1]{li2020online} that \(\phi(\x)\) is \(4\gamma\)-smooth over \(\Rn ^ d \times \dots \times \Rn ^ d\). 
Therefore, \begin{align} \label{def:smoothness}
    \norm{\nabla_{\x} \phi(\x) - \nabla_{\y}\phi(\y)} \le 4\gamma \norm{\x - \y},
\end{align}
for all \(\x, \y \in \Rn ^ d\). This implies the following:
\begin{align*}
    16\norm{\x - \y} ^ {2} &\stackrel{(a)}\ge \bigg(\norm{2(x_{1} - y_{1}) + (x_{2} - y_{2})} ^ 2 +  \sum_{i = 2} ^ {N - 1} \norm{2(x_{i} - y_{i}) - (x_{i - 1} - y_{i - 1}) - (x_{i + 1} - y_{i + 1})} ^ {2} \\
    & \hspace{20mm} + \norm{(x_{N} - y_{N}) - (x_{N - 1} - y_{N - 1})} ^ {2} \bigg),
\end{align*}
where to get \((a)\), we square both sides of \eqref{def:smoothness} and then use \eqref{eq:nabla_phi}. Therefore, bounding \eqref{eq:continue_from_here} using the bound obtained above, we get \begin{align*}
    \norm{\nabla_{\x} g(\x) - \nabla_{\y} g(\y)} ^ 2 \le 16\gamma ^ 2 d ^ 2\norm{\x - \y} ^ 2,
\end{align*} 
which implies that \(g(\x)\) is \(4\gamma d\)-smooth. In \textbf{E2}, we have \(H(\x) = \frac{1}{\sqrt{2}d} g(\x)\) which implies that \(H(\x)\) is \(2\sqrt{2}\gamma\)-smooth.
\end{document}